\documentclass[12pt]{amsart}
\def\newthm#1#2{\newtheorem{#1}[dummy]{#2}%
	\expandafter\def\csname#2\endcsname##1{\hyperref[#1:##1]{#2~\ref*{#1:##1}}}}

\usepackage{etoolbox}
\usepackage{tikz-cd}
\usepackage{amsmath}
\usepackage{amssymb}
\usepackage{amsthm}
\usepackage{setspace}
\usepackage{lipsum}
\usepackage{floatrow}
\usepackage{mathtools}
\usepackage{hyperref}
\usepackage{mathrsfs}

\patchcmd{\section}{\scshape}{\bfseries}{}{}
\makeatletter
\renewcommand{\@secnumfont}{\bfseries}
\makeatother

\newcommand{\Section}[1]{\hyperref[sec:#1]{Section~\ref*{sec:#1}}}
\newcommand{\Table}[1]{\hyperref[tab:#1]{Table~\ref*{tab:#1}}}
\newcommand{\eqn}[1]{\hyperref[eqn:#1]{(\ref*{eqn:#1})}}

\theoremstyle{plain}
\newtheorem{thm}{Theorem}[section]
\newtheorem{lemma}[thm]{Lemma}
\newtheorem{Prop}[thm]{Proposition}
\newtheorem{Cor}[thm]{Corollary}

\theoremstyle{definition}

\newtheorem{remark}[thm]{Remark}
\newtheorem{definition}[thm]{Definition}
\newtheorem{claim}[thm]{Claim}
\newtheorem{assumption}[thm]{Assumption}

\newtheorem*{thm*}{Main Theorem}

\numberwithin{equation}{section}

\newcommand{\sslash}{/\mkern-6mu/}

\renewcommand{\P}{{\mathbb P}}
\newcommand{\C}{{\mathbb C}}

\newcommand{\Q}{{\mathbb Q}}
\newcommand{\Z}{{\mathbb Z}}
\newcommand{\N}{{\mathbb N}}

\renewcommand{\O}{\mathcal{O}}

\usepackage{bm}

\newcommand{\ignore}[1]{}

\begin{document}

\title{Wall-crossing for quasimaps to GIT stack bundles}  

\author[Supekar]{Shidhesh Supekar}
\address{Department of Mathematics\\ Ohio State University\\ 100 Math Tower, 231 West 18th Ave. \\ Columbus,  OH 43210\\ USA}
\email{sidsupekar@gmail.com}

\author[Tseng]{Hsian-Hua Tseng}
\address{Department of Mathematics\\ Ohio State University\\ 100 Math Tower, 231 West 18th Ave. \\ Columbus,  OH 43210\\ USA}
\email{hhtseng@math.ohio-state.edu}

\begin{abstract}
We define the notion of $\epsilon$-stable quasimaps to a GIT stack bundle, and study the wall-crossing behavior of the resulting $\epsilon$-quasimap theory as $\epsilon$ varies.
\end{abstract}

\date{\today}


\maketitle

\tableofcontents

\setcounter{section}{-1}

\section{Introduction}
We work over $\C$.

This paper considers GIT stack bundles over a nonsingular projective variety $Y$. Roughly speaking, a GIT stack bundle over $Y$ is a fiber bundle over $Y$,
$$[E^{\text{ss}}/G]\to Y,$$ 
whose fibers are isomorphic to a given GIT stack quotient $[{V}^{ss}(\theta)/G]$. A precise definition of GIT stack bundles is given in equation (\ref{def:GIT_stack_bundle}) below. The notion of GIT stack bundles generalizes the notion of GIT bundles, which was previously studied in \cite{oh_2021} and \cite{CLS}. Toric stack bundles are examples of GIT stack bundles which are previously studied, see \cite{jiang2008} and \cite{JTY}.

A recent advance in Gromov-Witten theory of GIT (stack) quotients is the wall-crossing formula for $\epsilon$-quasimaps, proven in increasing generality in \cite{CKwallcross}, \cite{CK:wall1}, \cite{CJR},  \cite{CK:wall2} and in full generality in \cite{YZ}. For a rational number $\epsilon>0$, the notion of $\epsilon$-quasimaps to a GIT quotient is introduced in \cite{ckm2014stable} (which extends the case of target Grassmannians in \cite{toda_2011}) and extended to GIT stack quotient in \cite{CCK}. $\epsilon$-quasimaps can be used to define Gromov-Witten type theories by integrations against virtual fundamental classes of moduli spaces
$$Q_{g,n}^{\epsilon}([E^{ss}/G],\beta)$$
of stable $\epsilon$-quasimaps. For $\epsilon$ large, $\epsilon$-quasimaps coincide with stable maps and $\epsilon$-quasimap theory recovers Gromov-Witten theory in this case. For $\epsilon>0$ very small, $\epsilon$-quasimaps are simpler and the associated theory can often be explicitly analyzed. Therefore it is reasonable to hope for explicit results in Gromov-Witten theory by studying how $\epsilon$-quasimap theory changes as $\epsilon$ varies.  It turns out that $\epsilon$-quasimap theory only changes when $\epsilon$ passes through certain values, called the walls. The wall-crossing formula expresses such changes explicitly. 

In this paper, we define the notion of $\epsilon$-quasimaps to GIT stack bundles, generalizing the notion of ($0^+$-)quasimaps to GIT bundles defined in \cite{oh_2021}. Using this notion, we construct $\epsilon$-quasimap theory of a GIT stack bundle. We then establish a wall-crossing formula for $\epsilon$-quasimap theory of a GIT stack bundle, closely following the approach of \cite{YZ}. The main result of this paper is 
\begin{thm*}[see Theorem \ref{thm:wallcrossing} for more details\footnote{Note that the integration symbol $\int$ here actually means the pushforward by a non-constant morphism, see (\ref{notation:pushforward}).}]
\begin{equation*}
\begin{split}
	\int_{[Q^{\epsilon_{-}}_{g,n}([E^{ss}/G],\beta)]^{\text{vir}}}&1-\int_{[Q^{\epsilon_{+}}_{g,n}([E^{ss}/G],\beta)]^{\text{vir}}}1\\
	&=\sum_{k\geq 1}\sum_{\overline{\beta}}\dfrac{1}{k!}\int_{[Q^{\epsilon_{+}}_{g,n+k}([E^{ss}/G],\overline{\beta})]^{\text{vir}}}\prod_{i=1}^{k}\mathrm{ev}^*_{n+i}[(z-\psi_{n+i})I_{\tilde{\beta}_i}(z-\psi_{n+i})]_0	
\end{split}
\end{equation*}
where $\overrightarrow{\beta}=(\overline{\beta},\beta_1,\cdots,\beta_k)$ runs through all the $(k+1)$-tuples of effective curve classes such that $\beta=\overline{\beta}+\beta_1+\cdots+\beta_k$ and $\text{deg}(\beta_i)=(0,d_0)$ for $i=1,\cdots,k$.
\end{thm*}

We now discuss some future applications of the Main Theorem. We expect to prove a {\em Givental-style mirror theorem} for GIT stack bundles $[E^{\text{ss}}/G]\to Y$, which means that we expect to construct an explicit object using the Gromov-Witten theory of the base $Y$ and the geometry of the GIT stack bundle $[E^{\text{ss}}/G]\to Y$ and to prove that this object lies on Givental's Lagrangian cone for Gromov-Witten theory of $[E^{\text{ss}}/G]$. The Main Theorem explicitly relates Gromov-Witten theory of $[E^{\text{ss}}/G]$, corresponding to very large $\epsilon$, with $0^+$-quasimap theory of $[E^{\text{ss}}/G]$. It is expected that the geometry of $0^+$-quasimaps in genus $0$ is simpler and explicit calculations in genus $0$ are possible. We expect that this will lead to a desired mirror theorem\footnote{This approach was taken to derive mirror theorems for some GIT stacks in \cite{CCK}.}. We expect that the resulting mirror theorem for GIT stack bundles should recover known mirror theorems for toric bundles \cite{Brown2009GromovWittenIO}, toric stack bundles \cite{JTY}, and flag bundles \cite{oh_2021} (see also \cite{CLS}). For GIT stack bundles in general, with nonabelian group $G$, we plan to handle the calculation in $0^+$-quasimap theory by proving an abelian/nonabelian correspondence\footnote{In another direction, abelian/nonabelian correspondence for genus zero $0^+$-quasimap theory together with the Main Theorem should imply abelian/nonabelian correspondence for genus $0$ Gromov-Witten theory of $[E^{\text{ss}}/G]$.}. For GIT bundles, this is done in \cite{S} by following the approach of \cite{webb2021abeliannonabelian}. For GIT stack bundles, this will be studied in \cite{ST}.

We also expect that the genus $0$ mirror theorem discussed above and its consequences will be very useful in studying Virasoro constraints for GIT stack bundles. Virasoro constraints for toric bundles were studied in \cite{CGT}, where the results of \cite{Brown2009GromovWittenIO} were indispensable.

The Main Theorem in genus $0$ is expected to have another application, as follows. Given a smooth projective variety $X$, a smooth irreducible divisor $D\subset X$, and a natural number $r$, one can construct the stack $X_{D,r}$ of $r$-th roots of $X$ along $D$, see e.g. \cite[Appendix B.2]{AGV} for details. When $D$ is nef, genus $0$ Gromov-Witten invariants of $X_{D,r}$ are calculated in \cite{FTY} assuming knowledge about genus $0$ Gromov-Witten invariants of $X$. One geometric input used in \cite{FTY} is the expression of $X_{D,r}$ as the zero locus of a section of a line bundle on a $\mathbb{P}^1$-{\em stack} bundle over $X$. When $D$ is nef, quantum Lefschetz \cite{CCIT} is applied. When $D$ is not nef, quantum Lefschetz is not applicable. But one can still realize $X_{D,r}$ as a hypersurface in a GIT stack bundle using this expression. We hope to directly analyze the $0^+$-quasimap theory in genus $0$ and apply the Main Theorem to deduce results on Gromov-Witten invariants of $X_{D,r}$ when $D$ is not nef. If successful, we plan to further extend this to the case when the pair $(X,D)$ is simple normal-crossing, along the line of \cite{TY}.

The rest of the paper is organized as follows. In Section \ref{sec:prelim} we present some basic definitions, such as presentations of GIT stack bundles in Section \ref{secTS} and quasimap $I$-functions in Section \ref{IFunc}. In Section \ref{sec:e-quasimap} we discuss the notion of $\epsilon$-quasimaps to a GIT stack bundle and their moduli spaces. In Section \ref{subsec:defn_e-quasimap} we define $\epsilon$-quasimaps. In Section \ref{subsec:proper_e-quasimap} we prove that moduli stacks of $\epsilon$-quasimaps to a GIT stack bundle are proper. In Section \ref{obthr} we construct obstruction theories for moduli stacks of $\epsilon$-quasimaps. In Section \ref{sec2.9} we consider $\epsilon$-quasimaps to a specific GIT stack bundle. In Section \ref{sec:master_space} we consider the master spaces, which extend the objects introduced in \cite{YZ} to GIT stack quotients. After introducing various required notions, the definition of master spaces is given in Section \ref{subsec:defn_master_space}. Properness of master spaces is proven in Section \ref{subsec:proper_master_space}. In Section \ref{sec:localization}, we work out in detail the virtual localization formula for a $\C^*$-action on master spaces. Finally, in Section \ref{sec:wallcrossing}, we establish the wall-crossing formula.

\subsection{Acknowledgment}
We are very grateful to Jeongseok Oh and Yang Zhou for helpful discussions. We are very grateful to the referee for a very careful read and insightful comments. H.-H. T. is supported in part by Simons foundation collaboration grant.

\section{Preliminaries}\label{sec:prelim}
Throughout this paper, we consider the following set-up. Let $G$ be a reductive complex algebraic group. Let $Y$ be a non-singular projective variety and let 
$$\pi:E\rightarrow Y$$
be a fiber bundle over $Y$ with a fiberwise $G$-action such that the fiber $V$ is an affine variety with at worst local complete intersection singularities. Fix a character $\theta:G\rightarrow \C^*$ in $\chi(G)$ and $$L_{\theta}:=E\times \C_{\theta}\in \text{Pic}^G(E),$$ where $\C_{\theta}$ is the one dimensional representation of $G$ determined by character $\theta$. Let $$V^s(G,\theta), V^{ss}(G,\theta)\subset V$$ be the $\theta$-stable and $\theta$-semistable loci respectively.

\subsection{Presentations of GIT stack bundles}\label{secTS}
The notion considered in this Section is adopted from \cite{oh_2021}. We make the following
\begin{assumption}
\hfill
\begin{enumerate}
\item There is a morphism of varieties $\psi:Y\rightarrow \prod_{j=1}^{r}\P^{n_j-1}$ for some $r,n_j\in \Z_{>0}$.
\item There is a $S:=(\C^*)^r$ action on $V$ which commutes with the $G$-action, where $r$ is same as in the assumption above.
\end{enumerate}
\end{assumption}

\begin{definition}
A \textit{presentation} of $(E, G, \theta)$ is the triple $$(\psi:Y\rightarrow \prod_{j=1}^{r}\P^{n_j-1}, (S\times G)-\text{action on } V, m\in \Z_{>0})$$ which satisfies the following conditions:
\begin{enumerate}
	\item[(i)] $E$ is the pullback of a vector bundle $[(\prod_{j=1}^{r}\C^{n_j})\times V/S]$ on $[\prod_{j=1}^{r}\C^{n_j}/S]$. That is, we have a fiber product
	\begin{equation}\label{dia1}
		\begin{tikzcd}
			E \arrow{r} \arrow{d} & {[(\prod_{j=1}^{r}\C^{n_j})\times V/S]} \arrow{d} \\
			Y \arrow{r} & {[\prod_{j=1}^{r}\C^{n_j}/S]}.
		\end{tikzcd}		
	\end{equation}
	Here the bottom map is given by $$Y\xrightarrow{\psi}\prod_{j=1}^{r}\P^{n_j-1}\hookrightarrow \prod_{j=1}^{r}[\C^{n_j}/\C^*]\cong \left[\prod_{j=1}^{r}\C^{n_j}/S\right];$$
	\item[(ii)] The morphism $E\rightarrow [(\prod_{j=1}^{r}\C^{n_j})\times V/S]$ in diagram (\ref{dia1}) is $G$-equivariant;
	\item[(iii)] Let $A$ be the coordinate ring of the affine space $(\prod_{j=1}^{r}\C^{n_j})\times V$. Then, $A^{S\times G}\cong \C$ as $\C$-algebras;
	\item[(iv)] $V^s(G,\theta)=V^{ss}(G,\theta)\neq \emptyset$ and is nonsingular. Moreover, the $G$-action on $V^s(G,\theta)$ has finite stabilizers;
	\item[(v)] Denote $W:=\prod_{j=1}^{r}\C^{n_j}$ and $\tilde{V}:=W\times V$. Then 
	\[
	\tilde{V}^{ss}(S\times G,\tilde{\theta})=W^{ss}(S, \eta)\times V^{ss}(G,\theta).
	\]
	Here $\eta=(1,...,1)\in \Z^r\cong \chi(S)$ and $\tilde{\theta}=m\eta + \theta \in \chi(S)\oplus\chi(G)\cong\chi(S\times G)$.
\end{enumerate}
\end{definition}
From conditions (i) and (ii) above, we have a fiber diagram
\begin{equation}
\begin{tikzcd}\label{cd8}
{[E/G]} \arrow{r}\arrow["\pi"] {d}& {[\tilde{V}/(S\times G)]}\arrow{r}\arrow{d} &{[V/(S\times G)]}\arrow{d}\\
Y \arrow{r} & {[W/S]}\arrow{r}& BS.
\end{tikzcd}
\end{equation}
Here $\tilde{V}$ and $W$ are as defined in condition (v) and $BS:=[\text{Spec}\,\C/S]$ is the classifying space of $S$.

\begin{definition}
    $[E/G]\rightarrow Y$ is a fiber bundle over a base scheme $Y$ with stack quotients $[V/G]$ as its fibers. From diagram (\ref{cd8}) we have
    \[
    [E/G]:=[\tilde{V}/(S\times G)]\times_{[W/S]}Y.
    \]
\end{definition}

From conditions (iii) and (v), the GIT stack quotient of (semi-)stable locus $[\tilde{V}^{ss}(\tilde{\theta})/(S\times G)]$ is a nonsingular, projective and open substack of $[\tilde{V}/(S\times G)]$.

\begin{definition}
We define a {\em GIT stack bundle} over $Y$
\begin{equation}\label{def:GIT_stack_bundle}
    [E^{\text{ss}}/G]:=[\tilde{V}^{ss}(\tilde{\theta})/(S\times G)]\times_{[W^{ss}/S]}Y
\end{equation}
as a fiber bundle over $Y$ with GIT stack quotient $[{V}^{ss}(\theta)/G]$ as its generic fibers.
\end{definition}

From the definition, we have fibered diagram
\[
	\begin{tikzcd}
		{[E^{\text{ss}}/G]} \arrow{r}\arrow["\pi"] {d}& {[\tilde{V}^{ss}(\tilde{\theta})/(S\times G)]} \arrow{d}\\
		Y \arrow{r} & {[W^{ss}(\eta)/S]}.
	\end{tikzcd}
\]
Hence, $[E^{\text{ss}}/G]$ is nonsingular, projective and open substack of $[E/G]$. Moreover, analogous to \cite[Diagram 2.1]{CCK}, we have a diagram of natural morphisms for various GIT fiber bundles
\begin{equation}
    \begin{tikzcd}
 		{[E^{\text{ss}}/G]} \arrow[r,hook]\arrow[d]& {[E/G]}\arrow{d}\\
		{E\sslash_{\theta} G} \arrow{r} & {Y},
    \end{tikzcd}
\end{equation}
where $E\sslash_{\theta}G:=\tilde{V}\sslash_{\tilde{\theta}}(S\times G)\times_{W\sslash_{\eta}S}Y$ is a fiber bundle over $Y$ with the GIT scheme quotient $V\sslash_{\theta}G$ as its generic fibers.

\subsection{Graph Quasimaps and \textit{I}-functions.}\label{IFunc}
We recall some constructions for quasimaps to GIT (stack) quotients, see e.g. \cite{CKwallcross}.

Let $W$ be an affine variety with at worst local complete intersection singularities. Let $G$ be a reductive complex algebraic group acting on $W$ such that the (semi-)stable locus ($W^{\text{ss}}=W^{\text{s}}$) is smooth and non-empty. Denote
	\[
	\mathcal{X}:=[W^{\text{ss}}/G].
	\]
Let $\beta:\text{Pic}([W/G])\rightarrow \Q$ be an effective curve class and fix a $\epsilon\in \Q_{>0}$. 

 A prestable quasimap from a twisted marked curve $(C,x)$ 
\cite[Section 4]{AGV} to $\mathcal{X}$ consists of a tuple $((C,x),u)$ such that $u:C\rightarrow [W/G]$ is a representable morphism with finitely many base points disjoint from nodes and marked points.

	\begin{definition}
	The length $l(p)$ at a point $p$ of a prestable quasimap $((C,x_1,\cdots, x_n), [u])$ to a quotient stack target $\mathcal{X}$ is defined by \[l(x):=\text{min}\left\{\frac{(u^*s)_p}{m}\bigg\rvert \text{  }s\in H^0(W,L_{m\theta})^G,u^*s\neq 0,m>0\right\}.\]
	\end{definition}
	
\begin{remark}
We refer to \cite{ckm2014stable} for the definition of length at a point of a quasimap to a general GIT target. From the definition, $l(x)$ is non-zero if and only if $x$ is a base point. By prestability, these points are away from stacky nodes and markings. Hence, we can use the same notion of stability for orbifold target.
	\end{remark}

 Let $\varphi:(C,x_1,\cdots ,x_n)\rightarrow (\underline{C}, \underline{x}_1,\cdots ,\underline{x}_n)$ be the coarse moduli space for marked twisted curve $(C,x_1,\cdots ,x_n)$.

\begin{definition}\label{def5}
\cite[Section 2.3]{CCK} A prestable quasimap $((C, x_1,\cdots , x_k), [u])$ to a quotient stack $\mathcal{X}$ is said to be $\epsilon$-stable  if the following two conditions hold:
\begin{enumerate}
 \item The $\Q$ line bundle \[\omega_{\underline{C}}\left(\sum_{i=1}^{n}\underline{x_i}\right)\otimes (\varphi_*([u]^*{L}_{\theta}))^{\otimes\epsilon}\] on the coarse curve $\underline{C}$ is ample;
 \item $\epsilon l(x)\leq 1$ for every point $x$ in $C$.
\end{enumerate} 
A moduli of genus-$g$, $\epsilon$-stable quasimaps to $\mathcal{X}$ of curve class $\beta$ with $n$-marked points is denoted by $Q^{\epsilon}_{g,n}(\mathcal{X},\beta)$.
\end{definition}

\begin{definition}
A single marked, genus-$0$ \textit{($\epsilon$-)quasimap graph space} is defined to be a special moduli space of genus-$0$, ($\epsilon$-)stable quasimaps:
	\[
	QG_{0,1}(\mathcal{X},\beta):=Q^{\epsilon}_{0,1}(\mathcal{X}\times\P^1,\beta\times[\P^1]).
	\]
\end{definition}

Thus the underlying curve of an object in this graph space has a unique rational tail with its coarse moduli mapped isomorphically to $\P^1$. Consider a $\C^*$-action on $\P^1$ given  by
	\[
	\begin{tikzcd}
		{\lambda\cdot[x,y]=[\lambda x,y],}&{\lambda\in \C^*.}
	\end{tikzcd}
	\] 
This induces a $\C^*$-action on $QG_{0,1}(\mathcal{X},\beta)$. Let 
$$F_{\star,\beta}\subset QG_{0,1}(\mathcal{X},\beta)^{\C^*}$$ 
be a $\C^*$-fixed component such that the unique marked point is mapped to $\infty\in\P^1$ and $0\in\P^1$ is a base point of length deg$(\beta)$. Let $[F_{\star,\beta}]^{\text{vir}}$ be its virtual fundamental class and let $N^{\text{vir}}_{F_{\star,\beta}/QG_{0,1}(\mathcal{X},\beta)}$ be the virtual normal bundle in the sense of \cite{Graber1997LocalizationOV}. We denote the $\C^*$-equivariant parameter to be $z$ with the Euler class of standard representation as \[e_{\C^*}(\C_{\text{std}})=-z.\]

Let \[
I_\mu(\mathcal{X}) =\coprod_{r}I_{\mu_r}(\mathcal{X})
\]
denote the cyclotomic inertia stack of the stack $\mathcal{X}$ (c.f. \cite[Section 3.1]{AGV}, \cite[Section 1.5]{YZ}). Define
\[
\hat{\text{ev}}:QG_{0,1}(\mathcal{X},\beta)\rightarrow I_{\mu}(\mathcal{X})
\] 
to be the composition of evaluation map at the unique point with the band-inverting involution on $I_{\mu}(\mathcal{X})$.  

\begin{definition}
We define an $I$-function as
\begin{equation}
I(q,z):=1+\sum_{\beta>0}q^{\beta}I_{\beta}(z)
\end{equation}
where the sum is over all effective curve classes $\beta\in \text{Pic}([W/G])^{\vee}$ and 
	\[
	I_{\beta}(z):=(-z\mathfrak{r}^2)(\hat{\text{ev}})_*\left(\dfrac{[F_{\star,\beta}]^{\text{vir}}}{e_{\C^*}(N^{\text{vir}}_{F_{\star,\beta}/QG_{0,1}(\mathcal{X},\beta)})}\right).
	\]
Here $\mathfrak{r}$ is a locally constant function on $I_{\mu}(\mathcal{X})$ that takes the value $r$ on $I_{\mu_r}(\mathcal{X})$.
\end{definition}

\subsection{Inflated Projective bundle}\label{InfProjBun}
We recall the definition of inflated projective bundle and refer the reader to \cite[Appendix A]{YZ} for a more detailed treatment.
 
Let $X$ be any algebraic stack. Let $L_1, \cdots , L_k$ be line bundles on $X$. Consider a projective bundle 
\[
P=\P(L_1\oplus\cdots\oplus L_k)\rightarrow X,
\]
and coordinate hyperplanes
\[
H_i=\P(L_1\oplus\cdots\oplus\{0\}\oplus\cdots\oplus L_k),
\]
where $\{0\}$ is in the $i$-th place. Set $P_{k-1}=P$ and inductively for $i=k-1,\cdots, 1$ define union of codimension-$i$ coordinate subspaces
\[
Z_i=\bigcup (H_{j_1}\cap\cdots\cap H_{j_i}),
\]
where ${j_1,\cdots,j_i}$ runs through all subsets of $\{1,\cdots,k\}$ of size $i$. Let $Z_{(i)}\subset P_i$ be the proper transform of $Z_i$ and let
\[
P_{i-1}=\text{Bl}_{Z_{(i)}}P_i\rightarrow P_i
\]
be the blowup along $Z_{(i)}$. Then $$\mathcal{P}(L_1,\cdots, L_k):=P_0\rightarrow X$$ 
is called \textit{inflated projective bundle} associated to line bundles $L_1,\cdots,L_k$.

\section{\texorpdfstring{$\epsilon$}{e}-Stable Quasimaps}\label{sec:e-quasimap}
Throughout the paper we use twisted curves with balanced nodes and trivialized gerbe markings \cite[Section 4]{AGV}. A marking on a family of curve $\pi:C\rightarrow R$ is a closed substack $\Sigma\subset C$ in the (relative) smooth locus, together with sections $R\rightarrow \Sigma$ of $\pi|_{\Sigma}$, such that $\Sigma\rightarrow R$ is a gerbe banded by some $\mu_r$. 

\subsection{Definitions}\label{subsec:defn_e-quasimap}
\subsubsection{Prestable quasimaps}
Let $(E, G, \theta)$ be as in Section \ref{secTS}, with \[(\psi:Y\rightarrow \prod_{j=1}^{r}\P^{n_j-1}, (S\times G)-\text{action on }V, m\in \Z_{>0})\] as its presentation. Choose a class \[\beta:=(\beta',\tilde{\beta})\in \text{Ker}(\text{Pic}(Y)^{\vee}\oplus\text{Pic}^{S\times G}(\tilde{V})^{\vee}\rightarrow \chi(S)^{\vee}).\]
The map $\text{Pic}(Y)^{\vee}\oplus\text{Pic}^{S\times G}(\tilde{V})^{\vee}\rightarrow \chi(S)^{\vee}$ is the one induced from diagram (\ref{cd8}).

\begin{definition}\label{def2.1}
We define a \textit{prestable morphism} $C\rightarrow [E/G]$ of type $(g,n,\beta)$ to be a pair of morphisms
\[
(f:(C,x_1,\cdots, x_n)\rightarrow Y, u:C\rightarrow [\tilde{V}/(S\times G)])
\]
such that,
\begin{itemize}
\item $(C,x_1,\cdots, x_n)$ is a genus $g$ twisted curve with $n$ markings;
    
\item $f$ is a morphism of degree $\beta'$;
	
\item $u$ is a representable morphism of stacks. Moreover, the group homomorphism $$\tilde{\beta}:\text{Pic}([\tilde{V}/(S\times G)])\rightarrow \Q, \quad \tilde{\beta}(L):=\text{deg}(u^*(L))$$ is the class of map $u$;
	
\item $Q_f\cong Q_u$ as principal $S$-bundles over $C_0$.
\end{itemize}
\end{definition}

To define $C_0,Q_f$ and $Q_u$, we first note that by \cite[Proposition 9.1.1]{AV}, $u:C\rightarrow [\tilde{V}/(S\times G)]$ factors as
\begin{equation}\label{eq3}
	C\xrightarrow{\phi}C_0\xrightarrow{[u]}[\tilde{V}/(S\times G)],
\end{equation}
where $\phi: C\rightarrow C_0$ is a contraction of $u$-constant unmarked rational trees \cite[Lemma 9.2.1]{AV} (connected tree of rational curves attached to rest of the curve at single point) and $[u]$ is a representable morphism of class $\tilde{\beta}$. Note that factorization in \cite[Proposition 9.1.1]{AV} contracts both unstable rational trees and unstable rational bridges. In order to avoid contraction of rational bridges, we stabilize them by adding sections, apply contraction and forget the added sections. Composition of morphisms
\begin{equation}\label{eq4}
	C_0\xrightarrow{[u]} [\tilde{V}/(S\times G)]\rightarrow[W/S].
\end{equation}
provides a principal $S$-bundle on $C_0$ which we denote by $$Q_u.$$ 

We can similarly construct a principal $S$-bundle starting with the map $f$. To see this, consider a composition of morphisms 
\[
{C}\xrightarrow{f} Y\xrightarrow{\psi} \prod_{j=1}^r\P^{n_j-1}\rightarrow\P^{n_{j_0}-1}
\]
for each $1\leq j_0\leq r$. This is equivalent to a surjective morphism of sheaves on ${C}$, $$\O_{{C}}^{\oplus n_{j_0}}\xrightarrow{\varphi_{j_0}}\mathcal{L}_{j_0}\rightarrow 0,$$ 
where $\mathcal{L}_{j_0}$ is the pullback of $\O_{\P^{n_{j_0}-1}}(1)$ on ${C}$. Denote by $T_1,\cdots ,T_l$  the rational tails on each fiber $C$ contracted under $\phi$. Let $t_1,\cdots , t_l$ be the sections of ${C}_0$ corresponding to contraction points of $T_1,\cdots, T_l$ on $C_0$. Consider morphisms of sheaves on ${C}_0$, 
\[
\O_{{C}_0}^{\oplus n_{j_0}}\xrightarrow{\varphi_{j_0}|_{{C}_0}}\mathcal{L}_{j_0}|_{{C}_0}\hookrightarrow \mathcal{L}_{j_0}|_{{C}_0}\otimes \O_{{C}_0}(\Sigma_{i=1}^l\text{deg}(\mathcal{L}_{j_0}|_{T_i})\cdot t_i).
\]
This gives a morphism of stacks 
\[
{C}_0\rightarrow \P^{n_{j_0}-1} \hookrightarrow [\C^{n_{j_0}}/\C^*],
\]
with degree deg$(\mathcal{L}_{j_0})$ (see \cite[Section 3.2.2]{CKwallcross}). So we have a morphism of stacks 
\begin{equation}\label{eq5}
C_0\rightarrow\prod_{j=1}^{r}[\C^{n_j}/\C^*]=[W/S],
\end{equation}
of degree $d':=(\text{deg}(\mathcal{L}_{1}),\cdots, \text{deg}(\mathcal{L}_{r}))$. This map comes with a principal $S$-bundle over $C_0$, which we denote by $$Q_f.$$

\begin{remark}
Using the factorization (\ref{eq3}), we denote a prestable morphism $(f, u):C\rightarrow [E/G]$ as a triple
\[
(\phi:(C,x)\rightarrow(C_0,x), f:C\rightarrow Y, [u]:C_0\rightarrow [\tilde{V}/(S\times G)]),
\]
for rest of the paper.    
\end{remark}

\begin{definition}\label{def2}
	We call a prestable morphism $(\phi, f, [u]):C\rightarrow [E/G]$ a {\em prestable quasimap} to $[E^{ss}/G]$ of type $(g,n,\beta)$ if 
	 \begin{enumerate}
		\item $(\phi, f, [u]):C\rightarrow [E/G]$ is a prestable morphism of type $(g,n,\beta)$;
		\item\label{def2.2} $f$ is non-constant on each component of a rational tree of $C$ if and only if the rational tree is contracted under $\phi$;
		\item $[u]$ is a representable morphism such that $[u]^{-1}[\tilde{V}^{\text{us}}(\tilde{\theta})/(S\times G)]$ is finite and disjoint from all the nodes and marked points.
	\end{enumerate}
Here $$[\tilde{V}^{\text{us}}(\tilde{\theta})/(S\times G)]:=[\tilde{V}/(S\times G)]\backslash[\tilde{V}^{ss}(\tilde{\theta})/(S\times G)]$$ 
is the base locus and $y\in [\tilde{V}^{\text{us}}(\tilde{\theta})/(S\times G)]$ is called a {\em base point}. We denote the stack of prestable quasimaps to $[E^{ss}/G]$ of type $(g,k,\beta)$ by \[Q^{\text{pre}}_{g,k}([E^{ss}/G],\beta).\]
\end{definition}

\begin{remark}
    Condition (\ref{def2.2}) is essential to keep the irreducible components of the underlying curve bounded, as proved in Lemma \ref{lem2.11}. Also, see Remark \ref{rem2.6} and \ref{rem2.8}.
\end{remark}

\begin{remark}
    Note that the definition of a GIT stack bundle as the pullback of a vector bundle from the product of projective spaces is essential in constructing the contraction maps used in Definition \ref{def2.1}.
\end{remark}

\subsubsection{Stability condition}  
To simplify the notation we introduce 
\[X:=[\tilde{V}/(S\times G)],\hspace{0.5in} X^{ss}:=[\tilde{V}^{ss}/(S\times G)].\]

\begin{definition}\label{stabcon}
A prestable quasimap to $[E^{ss}/G]$, \[(\phi:(C,x_1,\cdots x_n)\rightarrow (C_0,x_1,\cdots ,x_n),f:C\rightarrow Y,[u]:C_0\rightarrow X),\] is $\epsilon$-stable if for each irreducible component $C'\subset C_0$, $f|_{C'}$ is a stable map to $Y$ or its associated quasimap $((C_0,x),[u])|_{C'}$ to $X^{ss}$ is $\epsilon$-stable in the sense of Definition \ref{def5}.
\end{definition}

\begin{remark}\label{rem2.6}
    By condition (\ref{def2.2}) of Definition \ref{def2}, any rational tail $C'\subset C$ that survives contraction is $f$-constant. Hence, by stability condition, $((C_0,x),[u])|_{C'}$ to $X^{ss}$ is $\epsilon$-stable. In particular, if $\epsilon$ is very small, all rational tails are forced to be contracted, giving us $0^+$-stable quasimaps. On the other hand, if $\epsilon$ is sufficiently large, any non-zero degree rational tail is allowed in $C$, giving us usual stable maps.
\end{remark}

\begin{remark}\label{rem2.8}
    If we fix $f: C\rightarrow Y$ to be constant, by Definition \ref{def2} and stability condition, moduli space of $\epsilon$-stable maps in Definition \ref{stabcon}  becomes moduli space of $\epsilon$-stable quasimaps to GIT stack quotient $[V^{ss}(\theta)/G]$. Similarly, if we fix $u:C\rightarrow [\tilde{V/(S\times G)}]$ to be constant, we get moduli space of stable maps to $Y$. 
\end{remark}
\begin{definition}
An isomorphism between two quasimaps (over $\C$) $$(\phi:(C, x)\rightarrow (C_0,x), f:C\rightarrow Y, [u])$$ and $$(\phi':(C',x')\rightarrow (C'_0,x'), f':C'\rightarrow Y, [v])$$ is a pair of isomorphism and $2$-isomorphism $$(\theta: (C,x) \xrightarrow{\sim} (C',x'), \delta: [u]\circ\phi\xrightarrow{\sim}[v]\circ\phi'\circ\theta)$$ such that
\begin{equation*}
f'\circ \theta =f,	\quad \quad \quad  \theta(x_i)=x'_i \quad \quad \quad \text{for all } i=1,\cdots , n.
\end{equation*}
\end{definition}
	
\subsubsection{Moduli stack}
The moduli space of $\epsilon$-stable quasimaps to $[E^{ss}/G]$ of type $(g,n,\beta)$ is denoted by $$Q^{\epsilon}_{g,n}([E^{ss}/G],\beta).$$ 
As the stability condition is a union of two open conditions, $Q^{\epsilon}_{g,n}([E^{ss}/G],\beta)$ is an open substack of $Q_{g,n}^{\text{pre}}([E^{ss}/G],\beta)$. 
	
\begin{Prop}\label{Prop2}
For $\epsilon>0$ fixed, the automorphism group of an $\epsilon$-stable quasimap is finite and reduced.
\end{Prop}
\begin{proof}
It is enough to prove this for rational components $C'\subset C$.  For $C'$ such that $f|_{C'}$ is constant, $(C_0,[u])|_{C'}$ is $\epsilon$-stable and the result follows from \cite[Section 2.4.2]{CCK}. Furthermore, if $f|_{C'}$ is non-constant for $C'\subset C$ then $(C',f_{C'})$ is a stable map to $Y$ and by prestability condition, $\phi(C')$ cannot be a rational tail. Hence $(C',[u]|_{C'})$ is a rational bridge which is either $0^+$-stable or is constant under $[u]$. In either cases the proposition follows by \cite[Theorem 1.4.1]{AV} and \cite[Section 2.4.2]{CCK}.
\end{proof}

\begin{thm}\label{thm9}
$Q_{g,n}^{\epsilon}([E^{ss}/G],\beta)$ is a Deligne--Mumford stack, locally of finite type over $\C$.
\end{thm}
\begin{proof}
By Proposition \ref{Prop2}, it suffices to prove that the stack is Artin. Moreover, as $\epsilon$-stability condition is an open condition, $Q_{g,n}^{\epsilon}([E^{ss}/G],\beta)$ is an open substack of $Q_{g,n}^{\text{pre}}([E^{ss}/G],\beta)$. Hence it is enough to show $Q_{g,n}^{\text{pre}}([E^{ss}/G],\beta)$ is an Artin stack, locally of finite type. Indeed, $Q_{g,n}^{\text{pre}}([E^{ss}/G],\beta)$ fits into the following fiber diagram of Artin stacks, locally of finite type over $\C$ (\cite[Section 2.4.1]{CCK},\cite[Theorem 1.4.1]{AV}): 
	\begin{equation}\label{cd1}
		\begin{tikzcd}
		{Q_{g,n}^{\text{pre}}([E^{ss}/G],\beta)}\arrow[r]\arrow[d]& {Q_{g,n}^{\text{pre}}(X^{ss},({d'},{\beta''}))}\arrow[d]\\
		{\mathfrak{M}_{g,n}^{tw,\text{pre}}(Y,\beta')}\arrow[r] & {Q_{g,n}^{\text{pre}}([W^{ss}/S],{d'})},
	\end{tikzcd}
	\end{equation}
where we write $\tilde{\beta}=(\tilde{d},\beta'')$ and $\mathfrak{M}_{g,n}^{tw,\text{pre}}(Y,\beta')$ is a stack of degree $\beta'$ prestable maps from genus $g$, $n$-pointed twisted curve to $Y$. Over geometric points, the above diagram looks as follows:
\[
\begin{tikzcd}
	{(\phi: (C,x) \rightarrow (C_0,x),f,[u])}\arrow[r, mapsto]\arrow[d, mapsto]& {((C_0,x),[u])}\arrow[d, mapsto]\\
	{((C,x), f)}\arrow[r, mapsto] & {((C_0,x),[u'])}.
\end{tikzcd}
\]
Here, $[u']$ is the map obtained in (\ref{eq4}). Hence, the vertical map is given by projection. The lower map is constructed as in (\ref{eq5}). Note here that the contraction map $C\rightarrow C_0$ is completely determined by the curve $C$ and map $f$, see conditions (2) in Definition \ref{def2}. Hence, the lower map is well-defined. The diagram is commutative by the definition of prestable morphism to $[E^{ss}/G]$, specifically we need $Q_f\cong Q_{u}$.  

To see that the diagram is cartesian, consider the morphism 
\[
\mathfrak{M}_{g,n}^{tw,\text{pre}}(Y,\beta')\times_{Q_{g,n}^{\text{pre}}([W^{ss}/S],{d'})} Q_{g,n}^{\text{pre}}(X^{ss},({d'},\beta''))\rightarrow Q_{g,n}^{\text{pre}}([E^{ss}/G],\beta)
\]
which maps
\[
(((C,x), f),((C_0,x),[u]))\mapsto (\phi: (C,x) \rightarrow (C_0,x),f,[u]).
\]
Therefore, we have the required result.
 \end{proof}

\subsection{Properness of moduli stack}\label{subsec:proper_e-quasimap}
For a class $\beta=(\beta',\tilde{\beta})$ and line bundle $L'\times \tilde{L}$ on $[E/G]$, define
\[
\beta(L'\times \tilde{L}):=\beta'(L')+\tilde{\beta}(\tilde{L}).
\]
Let $\O(1)$ be an ample line bundle on $Y$. Then,
\begin{equation}\label{linebun}
	L:=\pi^*(\O(1)\otimes \psi^*(\times_{j=1}^{r}\O_{\P^{n_j-1}}(m)))\otimes [E\times \C_{\theta}/G]
\end{equation}
is an ample line bundle on $[E/G]$. By \cite[Lemma 3.2.1]{ckm2014stable}, we have $\beta(L)\geq 0$ and $\beta(L)=0$ if and only if $\beta=0$, if and only if $f$ and $ [u] $ are constant. This gives the following result on boundedness. 

\begin{lemma}\label{lem2.11}
The number of irreducible components of the underlying curve of an $\epsilon$-stable quasimap to $[E^{ss}/G]$ of type $(g,n,\beta)$ is bounded. 
\end{lemma}
\begin{proof}
Irreducible components with positive genus or marked points are bounded by $g,n$. Rational tails $(g=0,n=0)$ contracted under $\phi$ are bounded because $\beta(L)>0$ on these components, since by definition $f$ is non-constant. Similarly, non-contracted rational components are bounded by applying \cite[Corollary 3.2.3]{ckm2014stable} (see \cite[Section 2.4.3]{CCK}) on quasimap $(C_0, [u])$ to $X^{ss}$.
\end{proof}

\begin{Cor}
  $Q^{\epsilon}_{g,n}([E^{ss}/G],\beta)$ is finite type over $\C$.  
\end{Cor}
\begin{proof}
This follows from Lemma \ref{lem2.11} and Theorem \ref{thm9}.
\end{proof}

\begin{Prop}\label{Prop1}
$Q_{g,n}^{\epsilon}([E^{ss}/G],\beta)$ is proper over $\C$.
\end{Prop}
Before we begin the proof, we start with an essential lemma.
\begin{lemma}\label{l6}
	$ Q^{\text{pre}}_{g,n}([E^{ss}/G],\beta )$ fits into the following fibered diagram,
	\begin{equation}\label{cd2}
		\begin{tikzcd}
			{Q_{g,n}^{\text{pre}}([E^{ss}/G],\beta)}\arrow[r]\arrow[d]& {Q_{g,n}^{\text{pre}}(X^{ss},({d'},\beta''))}\arrow[d]\\
			{\mathfrak{M}^{\text{pre}}_{g,n}(Y,\beta')}\arrow[r] & {Q_{g,n}^{\text{pre}}(W\sslash S,{d'})},
		\end{tikzcd}
	\end{equation}
where $\mathfrak{M}^{\text{pre}}_{g,n}(Y,\beta')$ is the stack of degree $\beta'$ prestable maps to $Y$ from $n$-pointed, genus $g$ non-twisted curves.
\end{lemma}
\begin{proof}
Over geometric points, the above diagram looks as follows 
	\[
	\begin{tikzcd}
		{(\phi: (C,x) \rightarrow (C_0,x),f,[u])}\arrow[r, mapsto]\arrow[d, mapsto]& {((C_0,x),[u])}\arrow[d, mapsto]\\
		{((\underline{C},x), \underline{f})}\arrow[r, mapsto] & {((\underline{C}_0,x),\underline{[u']})},
	\end{tikzcd}
	\]
where $\underline{C}$ and $\underline{C_0}$ are the coarse curves of $C$ and $C_0$ respectively.  $((\underline{C}_0,x),\underline{[u']})\in Q_{g,n}^{\text{pre}}(W\sslash S,{d'})$ is the coarse image of $((C_0,x),[u'])\in Q_{g,n}^{\text{pre}}([W^{ss}/S],{d'})$ from (\ref{cd1}). The vertical map on the right is given by forgetting the orbifold structure, while the bottom map is the contraction of rational tails on which $\underline{f}$ is nonconstant; see \cite[diagram 4.4.47]{oh_2021}. Then the lemma follows by an argument similar to the proof of (\ref{cd1}). The only nontrivial claim is $C\cong \underline{C}\times_{\underline{C}_0}C_0$. This follows by noting that rational trees contracted under $\phi$ do not carry any orbifold structure \cite[Lemma 9.2.1]{AV}.
\end{proof}

Now we have all the ingredients to prove properness.
\begin{proof}[Proof of Proposition \ref{Prop1}]
We use the valuative criteria. Let $R$ be a discrete valuation ring over $\C$ with quotient field $K$. Denote $\Delta:= \text{Spec}(R)$ and let $\bullet \in \text{Spec}(R)$ be its unique closed point. Set $\Delta_0:=\Delta\setminus\{\bullet\}=\text{Spec}(K)$.
	
We first prove separatedness. From Lemma \ref{l6},
	\[
	Q^{\epsilon}_{g,n}([E^{ss}/G],\beta)\subset \mathfrak{M}_{g,n}^{\text{pre}}(Y,\beta')\times_{Q_{g,n}^{\text{pre}}(W\sslash S,d')}Q^{\text{pre}}_{g,n}(X^{ss},(d',\beta'')).
	\]
	Let $\{((\phi_i,\underline{C}_i,f_i),((C_0)_i,[u'_i]))\}_{i=1,2}\in Q^{\epsilon}_{g,n}([E^{ss}/G],\beta)(\Delta)$ be two objects which are isomorphic over $\Delta_0$. We will show that they are isomorphic over $\Delta$.
	
	Recall that stability (Definition \ref{stabcon}) allows certain components of the underlying curve to be $\epsilon$-unstable if $(C,f)$ restricted to these components is a stable map to $Y$. Hence, even though our objects are $\epsilon$-stable, $((C_0)_i,[u'_i])$ may not be $\epsilon$-stable. Specifically, such a component $C_0'\subset C_0$ may fail $\epsilon$-stability in one (or more) of the following ways: 
	\begin{enumerate}
		\item $C_0'$ is an unmarked rational tail connected to the rest of the curve at one point with degree less than $1/\epsilon$,
		\item $C_0'$ is an unmarked rational bridge (connected to the rest of the curve at two points) with degree $0$,
		\item $C_0'$ has a base point of length greater than $1/\epsilon$.
	\end{enumerate}
	By condition (\ref{def2.2}) of Definition \ref{def2}, $\phi$ contracts rational tails on which $f$ is non-constant. Hence, by stability condition all rational tails that survive the contraction $\phi$ are $\epsilon$-stable. Hence situation (1) above does not occur. We add $l$ sections $s_i:\Delta\rightarrow C_i\xrightarrow{\phi_i}(C_{0})_{i}$ to these surviving rational tails and rational bridges so they become $0^+$-stable. Therefore we get $((\phi_i,\underline{C}_i,x_i,s_i, f_i),((C_0)_i,x_i,s_i,[u'_i]))\in Q^{0^+}_{g,n+l}([E^{ss}/G],\beta)(\Delta)$.
	
Now we note that $Q^{0^+}_{g,n+l}([E^{ss}/G],\beta)$ is separated. This follows by following the proof of \cite{oh_2021} and noting that the stable (quasi)map moduli $\overline{M}_{g,n}(Y,\beta')$ and  $Q^{0+}_{g,n}(X^{ss},(d',\beta''))$ are proper. Hence, $\{((\phi_i,\underline{C}_i,x_i,s_i, f_i),((C_0)_i,x_i,s_i,[u'_i]))\}_{i=1,2}$ are isomorphic over $\Delta$. Finally, forgetting the added extra sections, we see $ \{((\phi_i,C_i,f_i),((C_0)_i,[u'_i]))\}_{i=1,2} $ are isomorphic over $\Delta$, proving separatedness.

Next, to show completeness, we carefully modify the argument of \cite{oh_2021} to make it work for $\epsilon$-stable case. We take an object 
	\[
	((\phi,(C,x),f),((C_0,x),[u]))\in Q^{\epsilon}_{g,n}([E^{ss}/G],\beta)(\Delta_0)
	\]
	over $\Delta_0$ and our aim is to construct its extension over $\Delta$. 
	
We first construct a family of curves over $\Delta$ by extending stabilization of $C_0$. Note that $((C_0,x),[u])\in Q^{\text{pre}}_{g,n}(X^{ss},(d',\beta''))(\Delta_0)$ can be made $0^+$-stable by adding sections on the rational tails and contracting rational bridges which map constantly under $[u]$. As $Q^{0^+}_{g,n+l}(X^{ss},(d',\beta''))$ is proper, we can extend
	\[
	((C'_0,x,r),[u'])\in Q^{0^+}_{g,n+l}(X^{ss},(d',\beta''))(\Delta_0)
	\]
	to get
		\[
	((\overline{C}'_0,\overline{x},\overline{r}),[\overline{u}'])\in Q^{0^+}_{g,n+l}(X^{ss},(d',\beta''))(\Delta).
	\]
	
Next, we use the curve $\overline{C}'_0$ over $\Delta$ to add sections on $C$. Let $q_i:\Delta\rightarrow \overline{C}'_0$ be distinct sections such that $p_i:\Delta_0\xrightarrow{q_i\vert_{\Delta_0}} \overline{C}'_0\vert_{\Delta_0}\hookrightarrow C$ makes $((C,x,p),f)$ a family of stable maps over $\Delta_0$. As  $\overline{M}_{g,n+l}(Y,\beta')$ is proper, we can extend to get a family of stable maps over $\Delta$. That is, we get,
	\[
	((\overline{C},\overline{x},\overline{p}),\overline{f})\in \overline{M}_{g,n+l}(Y,\beta')(\Delta).
	\]
	Forgetting the sections $p_i$ and contracting $f$-nonconstant rational tails, we get a map $\overline{\phi}:(\overline{C},x)\rightarrow (\overline{C}_0,x)$. Moreover, this map extends the contraction map over $\Delta_0$ since $(\overline{C}_0,x)\vert_{\Delta_0}=(C_0,x)$.
	
	Now that we have a hold over curve $\overline{C}_0$ over $\Delta$ which extends $C_0$, we can start adding sections to stabilize $[u]$. Recall from the discussion above, all the surviving rational tails in $C_0$ are $\epsilon$-stable. However, $C_0$ may have base points of length greater than $1/\epsilon$. Let $m$ be the maximum length of base points. If $m<1/\epsilon$, we take $\epsilon'=\epsilon$. If $m\geq 1/\epsilon$, take $\epsilon'\in \Q$ such that $1/(m+1)<\epsilon'<1/m$. Our aim is to make $[u]:C_0\rightarrow X$ $\epsilon'$-stable. By definition, lengths of all base points are less than $1/\epsilon'$. If the degree $d$ of any rational tail satisfies $1/\epsilon<d\leq 1/\epsilon'$, we add a section on that component. Similarly, we add a section on any $[u]$-constant rational bridge. 
	
	Following the discussion above, let $q:\Delta\rightarrow \overline{C}_0$ be the added distinct sections such that,
	\[
	((C_0, x, q), [u])\in Q^{\epsilon'}_{g,n'}(X^{ss}, (d',\beta''))(\Delta_0).
	\]
	As $Q^{\epsilon'}_{g,n'}(X^{ss}, (d',\beta''))(\Delta_0)$ is proper, we get its extension
	\[
	((\overline{C}_0, \overline{x}, \overline{q}), [\overline{u}_0])\in Q^{\epsilon'}_{g,n'}(X^{ss}, (d',\beta''))(\Delta).
	\]
	
The data $((\overline{\phi}, (\overline{C},\overline{x}),\overline{f}),((\overline{C}_0,\overline{x}),[\overline{u}_0]))$ form a candidate for the extension over $\Delta$. However, the central fiber may still have some problematic rational components which may map constantly under both $\overline{f}$ and $[\overline{u}_0]$. More importantly, as $((\overline{C}_0, \overline{x}, \overline{q}), [\overline{u}_0])$ is an $\epsilon'$-stable extension, we may have base points of length greater than $1/\epsilon$ on $\overline{f}$-constant rational components. 

We first contract the problematic rational tails. To be precise, contract rational tails  $C''$ (if any) in the central fiber such that map $\overline{f}\vert_{C''}$  and $[[\overline{u}_0]\circ\overline{\phi}]\vert_{C''}$ are constant (using the line bundle (\ref{linebun}) to stabilize). Let $\tilde{C}$ and $\tilde{C}_0$ be the respective curves after such contraction and $\tilde{\phi}$ the contraction map between them. Define $\tilde{f}$ and $[\tilde{u}_0]$ such that the following diagrams commute,
\[
\begin{tikzcd}
	{\overline{C}}\arrow[r] \arrow[d,"\overline{f}"'] & {\tilde{C}}\arrow[dl, "\tilde{f}"] &  & {\overline{C}_0}\arrow[r] \arrow[d,"{[\overline{u}_0]}"'] & {\tilde{C}_0}\arrow[dl, "{[\tilde{u}_0]}"] \\
	{Y} &  & & {X}.
\end{tikzcd}
\]

Finally, we tweak the map $[\tilde{u}_0]$ (on the central fiber) to make the extension $\epsilon$-stable. Let $((\tilde{\phi}, (\tilde{C},\overline{x}),\tilde{f}),((\tilde{C}_0,\overline{x}),[\tilde{u}_0]))$ be the object obtained from the above procedure. By prestability, base points are away from nodes and marked points. Hence, we can assume the underlying curve $\tilde{C}(\Delta_0)$ is smooth over $\Delta_0$ (shrinking $\Delta$ if necessary). After extension (and forgetting the stabilizing sections), curve over the central fiber $\tilde{C}(\bullet)$ may be nodal. 

Now we divide the argument into two cases.

\noindent\textit{Case 1}: If the map $(\tilde{C}(\bullet),\tilde{f}(\bullet))$ over the closed point $\bullet$ is stable, define \[((\hat{\phi}, (\hat{C},\hat{x}),\hat{f}),((\hat{C}_0,\overline{x}),[\tilde{u}])):=((\tilde{\phi}, (\tilde{C},\overline{x}),\tilde{f}),((\tilde{C}_0,\overline{x}),[\tilde{u}_0])).\] 

\noindent\textit{Case 2}: If the map $(\tilde{C}(\bullet),\tilde{f}(\bullet))$ from the central fiber is unstable, properness of $\overline{M}_{g,n}(Y,\beta')$ implies $(\tilde{C}(\Delta_0),\tilde{f}(\Delta_0))$ is also not stable. By stability, this means that the contraction map $\tilde{\phi}(\Delta_0)$ is an isomorphism and $(\tilde{C}_0(\Delta_0), [\tilde{u}_0](\Delta_0))$ is $\epsilon$-stable. As $Q^{\epsilon}_{g,n}(X^{ss},(d',\beta''))$ is proper, by extending we get $(\hat{C}_0,[\hat{u}])\in Q^{\epsilon}_{g,n}(X^{ss},(d',\beta''))(\Delta)$. Define $\hat{\phi}:\hat{C}\cong\hat{\underline{C}}_0$ (coarse curve of $\hat{C}_0$). Note that $\hat{C}_0$ may have extra components not present in $\tilde{C}$ (extension may add rational component of degree $1/\epsilon<d<1/\epsilon'$ over central fiber). Define $\hat{f}:\hat{C}\rightarrow Y$ such that the following diagram commutes,
 \[
 \begin{tikzcd}
 	{\hat{C}}\arrow[r] \arrow[d,"\hat{f}"'] & {\tilde{C}}\arrow[dl, "\tilde{f}"] \\
 	{Y},
 \end{tikzcd}
 \]
where the top map is the contraction of added rational tails over the closed point $\bullet$. 

\begin{claim} 
$((\hat{\phi}, (\hat{C},\hat{x}),\hat{f}),((\hat{C}_0,\hat{x}),[\hat{u}]))\in Q^{\epsilon}_{g,n}([E^{ss}/G],\beta)(\Delta)$ is the required extension. 
\end{claim}

Indeed, by construction $((\hat{\phi},(\hat{C},\hat{x}),\hat{f}),((\hat{C}_0,\hat{x}),[\hat{u}]))\vert_{\Delta_0}=((\phi, (C,x),f),(C_0, x),[u])$. We start by checking that $\hat{\phi}$ is the contraction of $\hat{f}$-nonconstant rational tails. Let $C''(\Delta)$ be a component of $\hat{C}(\Delta)$ such that over the central point $C''(\bullet) $ is a $\hat{f}(\bullet)$-nonconstant tail. This implies $C''(\Delta_0)$ is $\hat{f}(\Delta_0)$-nonconstant (shrinking if necessary) and $\hat{\phi}(\Delta_0)$ contracts the tail. Hence, by extension $\hat{\phi}(\bullet)$ contracts $C''(\bullet)$ too. Likewise, if the central fiber $C''(\bullet) $ is $\hat{f}(\bullet)$-constant and $C''(\Delta_0)$ is $\hat{f}(\Delta_0)$-constant, we get that $\hat{\phi}(\Delta)\vert_{C''}$ is an isomorphism. The only non-trivial case is when $C''(\Delta_0)$ is $\hat{f}(\Delta_0)$-nonconstant but its limit $C''(\bullet) $ is $\hat{f}(\bullet)$-constant. Here as $C''(\Delta_0)$ is $\hat{f}(\Delta_0)$-nonconstant, $\hat{\phi}(\Delta_0)$ contracts $C''(\Delta_0)$ and by extension $\hat{\phi}(\bullet)$ contracts $C''(\bullet)$. Hence, $C''(\bullet)$ is $[[\hat{u}]\circ\hat{\phi}](\bullet)$-constant. By assumption, $C''(\bullet)$ is also $\hat{f}(\bullet)$-constant. Such tails cannot occur as these are exactly the rational tails that we contract to get $\tilde{C}$ from $\overline{C}$. Therefore we have \[((\hat{\phi},(\hat{C},\hat{x}),\hat{f}),((\hat{C}_0,\hat{x}),[\hat{u}]))\in\mathfrak{M}_{g,n}^{\text{pre}}(Y,\beta')\times_{Q_{g,n}^{\text{pre}}(W\sslash S,d')}Q^{\text{pre}}_{g,n}(X^{ss},(d',\beta''))(\Delta).\] 

It remains to check $\epsilon$-stability. Let $C''$ be an irreducible component of $\hat{C}$ and let $C''_0$ be its corresponding image under $\hat{\phi}$. If the curve over central fiber $C''(\bullet)$ is $\hat{f}(\bullet)$-stable, we are done. So let us assume that $C''(\bullet)$ is $\hat{f}(\bullet)$-unstable. We want to show that $C''_0(\bullet)$ is $\epsilon$-stable. As $C''(\bullet)$ is $\hat{f}(\bullet)$-unstable, $C''(\Delta_0)$ is $\hat{f}(\Delta_0)$-unstable, then $C''_0(\Delta_0)$ is $\epsilon$-stable and by completeness of $Q^{\epsilon}_{g,n}(X^{ss},(d',\beta''))$, we get that the central fiber $C''_0(\bullet)$ is $\epsilon$-stable.
\end{proof}

\subsection{Obstruction Theory}\label{obthr}
	Let $\mathfrak{M}_{g,n}^{tw}$ be the stack of $n$-pointed genus $g$ twisted curves. It is a smooth Artin stack, locally of finite type \cite[Theorem 1.10]{olsson_2007}. Let $\mathfrak{C}\rightarrow \mathfrak{M}_{g,n}^{tw}$ be its universal curve. Define 
	\[
	\mathfrak{Bun}_{G}^{tw}:=\text{Hom}_{\mathfrak{M}_{g,n}^{tw}}(\mathfrak{C}, BG\times \mathfrak{M}_{g,n}^{tw}),
	\]
	which is smooth over $\mathfrak{M}_{g,n}^{tw}$ \cite[Section 2.4.5]{CCK}. We have a fiber diagram of forgetful morphisms,
	\[
	\begin{tikzcd}
		{Q^{\text{pre}}_{g,n}(X^{ss},(d',\beta''))}\arrow[r]\arrow[d]&{Q^{\text{pre}}_{g,n}([V^{ss}/(S\times G)],\beta'')}\arrow[d]\\
		{Q^{\text{pre}}_{g,n}([W^{ss}/S],d')}\arrow[r]&{\mathfrak{Bun}_{S}^{tw}}.
	\end{tikzcd}
	\]
	Combining this with diagram (\ref{cd1}), we get the following fibered diagram
	\begin{equation}
			\begin{tikzcd}
			{Q^{\text{pre}}_{g,n}([E^{ss}/G],\beta)}\arrow[r,"p_1"]\arrow[d,"p_2"]&{Q^{\text{pre}}_{g,n}([V^{ss}/(S\times G)],\beta'')}\arrow[d,"\mu_1"]\\
			{\mathfrak{M}_{g,n}^{tw}(Y,\beta')}\arrow[r,"\mu_2"]&{\mathfrak{Bun}_{S}^{tw}}.
		\end{tikzcd}
	\end{equation}
The map $\mu_1$ can be factored as 
\[
Q^{\text{pre}}_{g,n}([V^{ss}/(S\times G)],\beta'')\xrightarrow{\mu_1^1}\mathfrak{Bun}_{S\times G}^{tw}\xrightarrow{\mu_1^2}\mathfrak{Bun}_{S}^{tw}.
\]
Let the complex $E^{\bullet}_1\in D^b(Q^{\text{pre}}_{g,n}([V^{ss}/(S\times G)],\beta''))$ be the $\mu_1^1$-relative perfect obstruction theory, see \cite[Section 2.4.5]{CCK} for its construction. Define a complex (see e.g. \cite[Remark 4.5.3]{ckm2014stable})
\[
E^{\bullet}_{\mu_1}:=\text{Cone}(E^{\bullet}_1[-1]\rightarrow \mathbb{L}_{\mu_1^1}[-1]\rightarrow (\mu_1^1)^*\mathbb{L}_{\mu_1^2})\in D^b(Q^{\text{pre}}_{g,n}([V^{ss}/(S\times G)],\beta'')),
\]
where $\mathbb{L}_{\mu_1^1}, \mathbb{L}_{\mu_1^2}$ are cotangent complexes for $\mu_1^1,\mu_1^2$ respectively.

Next consider the composition
\[
\mu':\mathfrak{M}_{g,k}^{tw}(Y,\beta')\xrightarrow{\mu_2} \mathfrak{Bun}_{S}^{tw}\xrightarrow{\mu_3} \text{Spec}(\C).
\]
This composition factors through the stack of twisted curves via a forgetting map,
\[
\mathfrak{M}_{g,k}^{tw}(Y,\beta')\xrightarrow{\mu'_2} \mathfrak{M}_{g,k}^{tw}\xrightarrow{\mu'_3} \text{Spec}(\C).
\]
Let $E'^{\bullet}_2\in D^b(\mathfrak{M}_{g,k}^{tw}(Y,\beta'))$ be the relative perfect obstruction theory for $\mu'_2$. Such an obstruction theory exists by \cite[Section 4.5]{AGV}. Define 
\[
E^{\bullet}_{\mu'}:=\text{Cone}(E^{\bullet}_2[-1]\rightarrow \mathbb{L}_{\mu'_{2}}[-1]\rightarrow (\mu'_2)^*\mathbb{L}_{\mathfrak{M}_{g,k}^{tw}})\in D^b(\mathfrak{M}_{g,k}^{tw}(Y,\beta')).
\]
Using this, a perfect obstruction theory $E^{\bullet}_{\mu_2}$  relative to $ \mathfrak{Bun}_{S}^{tw} $ can be constructed using an argument similar to \cite{oh_2021}. We present the details in our situation. 

Let $\mathfrak{C}$ be the universal curve on $\mathfrak{M}_{g,k}^{tw}(Y,\beta')$. Let $\Phi:\mathfrak{C}\rightarrow \mathfrak{C_0}$ be the contraction of $f$-nonconstant rational tails. Similar to \cite{oh_2021}, we have a commutative diagram
\[
\begin{tikzcd}
	{T_{\mathfrak{C}}(-\sum_{i}p_i)}\arrow[r]\arrow[d,"\nu_1"']&{f^*T_Y}\arrow[d, "\nu_2"]\\
	{\Phi^*T_{\mathfrak{C_0}}(-\sum_{i}p_i)} \arrow[r]& {\Phi^*T_{BS}\vert_{\mathfrak{C_0}}}.
\end{tikzcd}
\]
Define the complex $E'^{\bullet}$ to be the dual derived pushforward of 
\[
\text{Cone}(\nu_1)[-1]\rightarrow f^*T_Y.
\]
As the composition $\text{Cone}(\nu_1)[-1]\rightarrow f^*T_Y\xrightarrow{\nu_2}  \Phi^*T_{BS}\vert_{\mathfrak{C_0}}$ is zero in the derived category, we have a morphism 
\[
\text{Cone}(\text{Cone}(\nu_1)[-1]\rightarrow f^*T_Y)\rightarrow \Phi^*T_{BS}\vert_{\mathfrak{C_0}}.
\]
This induces the dual morphism $\mu_2^*\mathbb{L}_{\mathfrak{Bun}_{S}^{tw}}\rightarrow E'^{\bullet}$ such that the following diagram commutes,
\[
\begin{tikzcd}
	{\mu_2^*\mathbb{L}_{\mathfrak{Bun}_{S}^{tw}}}\arrow[r]\arrow[d]&{E'^{\bullet}}\arrow[dl]\\
	{\mathbb{L}_{\mathfrak{M}_{g,k}^{tw}(Y,\beta')}}.
\end{tikzcd}
\]
Define a complex
\[
E^{\bullet}_{\mu_2}:=\text{Cone}(\mu_2^*\mathbb{L}_{\mathfrak{Bun}_{S}^{tw}}\rightarrow E'^{\bullet})\in D^b(\mathfrak{M}_{g,k}^{tw}(Y,\beta')).
\]
We define 
\[
E^{\bullet}:=(p_1^*E^{\bullet}_{\mu_1}\oplus p_2^*E^{\bullet}_{\mu_2})\vert_{Q^{\epsilon}_{g,k}([E/G],\beta)}\in D^b(Q^{\epsilon}_{g,k}([E^{ss}/G],\beta))
\]
as the relative perfect obstruction theory for $Q^{\epsilon}_{g,k}([E^{ss}/G],\beta)$ over $\mathfrak{Bun}_{S}^{tw}$.

\begin{remark}
When $\epsilon$ is large, we have seen that $Q^{\epsilon}_{g,k}([E^{ss}/G],\beta)$ is isomorphic to a moduli stack $\mathfrak{M}^{tw}_{g,k}([E^{ss}/G],\beta)$ of stable maps to $[E^{ss}/G]$. In this case, the {\em absolute} perfect obstruction theory obtained from $E^\bullet$ above can be seen to coincide with the usual {\em absolute} perfect obstruction theory for stable maps \cite{AGV}, using the short exact sequence relating the tangent bundles $T_{[E^{ss}/G]}$ and $T_Y$.
\end{remark}

\subsection{Evaluation map}\label{evmaps} Let $I_{\mu}(X^{ss})$ be the cyclotomic inertia stack of $X$. Define 
\[
I_{\mu}([E^{ss}/G]):=Y\times_{[W^{ss}/S]}I_{\mu}(X^{ss})
\]
to be the corresponding stack for our target space. To define an evaluation map from $Q^{\epsilon}_{g,k}([E^{ss}/G],\beta)$ to $I_{\mu}([E^{ss}/G])$, consider the universal curve $\mathcal{C}$  over $Q^{\epsilon}_{g,k}([E^{ss}/G],\beta)$ and the universal morphism $(\mathcal{F},\mathcal{U}):\mathcal{C}\rightarrow [E/G]$. Restriction of these maps to the marked points gives the desired evaluation map, which we represent as a pair 
\[
ev:=(ev',\tilde{ev}):Q^{\epsilon}_{g,k}([E^{ss}/G],\beta)\rightarrow I_{\mu}([E^{ss}/G]).
\]
Note that since marked points are away from base points, evaluation maps take values in $I_\mu([E^{\text{ss}}/G])$.

\subsection{Quasimaps to \boldmath \texorpdfstring{$\tilde{\P}$}{P}} \label{sec2.9}
We end this section with moduli of $\epsilon$-stable quasimaps to a space with $\P^N$ as its fibers and construct contraction maps on them. 

Let $Y$ be a nonsingular projective variety as before. For $N\geq 0$, consider an example where $V=\C^{N+1}$ and $G=\C^*$. Then $E\rightarrow Y$ is a total space with $\C^{N+1}$ as its fibers. Let ($\psi:Y\rightarrow \prod_{j=1}^r\P^{n_j},(\C^*)^r\times \C^*$-action on $\C^{N+1}, m\in\Z_{>0}$) be its presentation \cite[Section 4.2]{oh_2021}. Let $W:= \prod_{j=1}^r\C^{n_j+1}$ and $\tilde{V}=W\times \C^{N+1}$ be as before. Define $$\tilde{\P}:=[E^{ss}/G]\rightarrow Y.$$ 
The fibers of this map are $[(\C^{N+1}\backslash\{\overrightarrow{0}\})/\C^*]$. 

Consider the moduli stack $Q^{\epsilon}_{g,n}(\tilde{\P},\beta)$ of $\epsilon$-stable quasimaps to $\tilde{\P}$ of type $(g,n,\beta)$. From above, it is a proper Deligne--Mumford stack of finite type with a virtual fundamental class denoted by $[Q^{\epsilon}_{g,n}(\tilde{\P},\beta)]^{\text{vir}}$. 

For a positive integer $d_0$, let $$\epsilon_0=1/d_0$$ be a wall. Denote the chamber to the right as $$\epsilon_{+}:=(1/d_0,1/(d_0-1))$$ and similarly the chamber to the left as $$\epsilon_{-}:=(1/(d_0+1),1/d_0).$$ Consider a closed point in $Q^{\epsilon_+}_{g,n}(\tilde{\P},\beta)$,
\[
(\phi:C\rightarrow C_0,f:C\rightarrow Y,[u]:C_0\rightarrow [\tilde{V}/((\C^*)^r\times \C^*)]).
\]
Note that by condition (v) from the definition of presentation for $\tilde{\P}$ in Section \ref{secTS}, 
\[
\tilde{V}^{\text{ss}}((\C^*)^r\times \C^*,\tilde{\theta})=\prod_{j=1}^r(\C^{(n_j+1)}\backslash\{0\})\times (\C^{N+1}\backslash\{0\}).
\]
Thus we get a map from (semi-)stable locus to product of $ \P^n $'s;
\[
[\tilde{V}^{\text{ss}}/(S\times G)]\rightarrow \prod_{j=1}^r\P^{n_j}\times \P^N.
\]
From the stability condition, the map $[u]$ restricted to any rational tail of $C_0$ is $\epsilon_{+}$-stable. Moreover, $f$ is constant on the preimage of these tails in $C$. Consider contraction $C_0\rightarrow C'_0$ of rational tails having degree $(0,d_0)$ \cite[Section 3.2.2]{CKwallcross}, \cite{toda_2011}, \cite{MM}. The image of such a contraction map 
\[
(\phi':C\rightarrow C'_0,f:C\rightarrow Y,[u']:C'_0\rightarrow [W\times\C^{N+1}/((\C^*)^r\times \C^*)])
\] 
lies in $Q^{\epsilon_-}_{g,n}(\tilde{\P},\beta)$. This can be extended functorially over families of quasimaps to give a map
\begin{equation}
	c:Q^{\epsilon_+}_{g,n}(\tilde{\P},\beta)\rightarrow Q^{\epsilon_-}_{g,n}(\tilde{\P},\beta)
\end{equation}
By taking composition of such morphism, we have
\begin{equation}\label{eq9}
	c_{\epsilon}:Q^{\epsilon}_{g,n}(\tilde{\P},\beta)\rightarrow Q^{0+}_{g,n}(\tilde{\P},\beta).
\end{equation}
for every $\epsilon>0$.

Similarly, we can replace the last $k$ marked points by base points of length $d_0$  \cite[Section 3.2.3]{CKwallcross}. This map may convert certain rational components into rational tails. Hence, we take the $0^+$-stabilization of the obtained quasimaps. Let $C'\subset C$ be such a rational tail connected to the rest of the curve at $p$. Let $C'_0\subset C_0$ be its corresponding image under the contraction map $\phi$. If $f|_{C'}$ is non-constant, then we contract $C_0'$ from $C_0$ and replace $p$ by a base point of length equal to the degree of $C_0'\subset C_0$. If $f|_{C'}$ is constant, we contract both $C'\subset C$ and $C'_0\subset C_0$. Latter contraction gives a base point of length equal to the degree of $C_0'$. In short, for every $\epsilon>0$, we have a map
\begin{equation}\label{eq11}
b_{k,\epsilon}:Q^{\epsilon}_{g,n+k}(\tilde{\P},\beta)\rightarrow Q^{0+}_{g,n}(\tilde{\P},\beta+k(0,d_0)).
\end{equation}


\section{Master Spaces}\label{sec:master_space}
The purpose of this Section is to construct master spaces. With the goal of studying the behavior as the stability parameter crosses a wall $\epsilon_0=1/d_0$, master spaces should include $\epsilon_\pm$-stable quasimaps, where $\epsilon_-<\epsilon_0<\epsilon_+$. Roughly speaking, the effect of crossing the wall $\epsilon_0$ from the right to the left is to replace degree-$d_0$ rational tails by length-$d_0$ base points. Therefore, master spaces should include quasimaps with both degree-$d_0$ rational tails and length-$d_0$ base points. The way to control the transition from rational tails to base points and to avoid objects with infinitely many automorphisms, as discovered in \cite{YZ}, is to used the notion of {\em entangled tails}. We refer to \cite[Section 1.13]{YZ} for a more in-depth discussion about this.

\subsection{Weighted Twisted Curves} 
Let $d_0,\epsilon_+,\epsilon_-$ be as above. Fix non-negative integers $g,n$ and a pair of non-negative integers $d:=(d',d'')$, such that $2g-2+n+\epsilon_0d''\geq 0$. When $g=n=0$, we require the inequality to be strict, i.e. $\epsilon_0d''>2$. For $n$-marked twisted curves of genus $g$, we assign to each irreducible component $C$ a pair of non-negative integers $(d'_C,d''_C)$ such that $\sum_C(d'_C,d''_C)=(d',d'')$. We refer to $(d',d'')$ as the degree of the curve. We will see that such an assignment is continuous in a sense that over a family of curves it is locally given by the degree of some line bundle. Such a decorated twisted curve is referred as $n$-marked weighted twisted curve of genus $g$ and degree $d$ \cite{Cos}, \cite{HL10}. The moduli stack of such weighted curves is denoted by $\mathfrak{M}^{\text{wt}}_{g,n,d}$.

\begin{definition}
	We define the stack of \textit{$(\epsilon_0)$-semistable weighted curves} to be the open substack $\mathfrak{M}^{\text{wt},\text{ss}}_{g,n,d}\subset \mathfrak{M}^{\text{wt}}_{g,n,d}$ defined by the following conditions:
	\begin{itemize}
		\item the curve has no degree-$(0,0)$ rational bridge,
		\item the curve has no rational tail of degree strictly less than $(0,d_0)$ or strictly greater than $(0,\infty)$ in dictionary (lexicographic) order.
	\end{itemize}
\end{definition}

\subsection{Entangled tails} 
Here we discuss entangled tails briefly and refer readers to \cite[Section 2.2]{YZ} for a more detailed treatment. Let $m:=\lfloor d''/d_0\rfloor$ be the maximum number of degree-$(0,d_0)$ rational tails. Set 
\[ \mathfrak{U}_m=\mathfrak{M}^{\text{wt},\text{ss}}_{g,n,d}.\]
Let $\mathfrak{Z}_i\subset \mathfrak{U}_m$ be the reduced substack parameterizing curves with at least $i$ rational tails of degree $(0,d_0)$. $\mathfrak{Z}_m$ is the deepest stratum and hence smooth. Define \[\mathfrak{U}_{m-1}\rightarrow \mathfrak{U}_{m}\] 
to be the blowup along $\mathfrak{Z}_m$ and let $$\mathfrak{C}_{m}\subset \mathfrak{U}_{m-1}$$ be the exceptional divisor. Inductively let $\mathfrak{Z}_{(i)}\subset \mathfrak{U}_i$ be the proper transform of $\mathfrak{Z}_i$ and let 
\[\mathfrak{U}_{i-1}\rightarrow \mathfrak{U_i}\] 
be the blowup along $\mathfrak{Z}_{(i)}$. Set $\tilde{\mathfrak{M}}_{g,n,d}=\mathfrak{U}_0$.

\begin{definition}
	$\tilde{\mathfrak{M}}_{g,n,d}$ is called the moduli stack of \textit{genus $g$, $n$-marked, $\epsilon_0$-semistable curves of degree $d$ with entangled tails}.
\end{definition}

For a scheme $R$, an $R$-family of semistable curves with entangled tails corresponding to a morphism $e:R\rightarrow \tilde{\mathfrak{M}}_{g,n,d}$ is denoted by 
\[(\pi:C\rightarrow R, x,e),\] 
where $(\pi:C\rightarrow R,x)$ is the family of $n$-marked weighted twisted curves associated with $R\xrightarrow{e} \tilde{\mathfrak{M}}_{g,n,d}\rightarrow\mathfrak{M}_{g,n,d}^{\text{wt},\text{ss}}$.

Let 
\[\xi=(\pi:C\rightarrow\text{Spec}K,x,e)\]
be a geometric point of $\tilde{\mathfrak{M}}_{g,n,d}$. Let $E_1,\cdots,E_l$ be the degree $(0,d_0)$ rational tails in $C$. Locally let $\mathfrak{h}_i$ be the locus where $E_i$ remains a rational tail. Then near the image of $\xi$, the analytic branches of $\mathfrak{Z}_1\subset\mathfrak{M}^{\text{wt},\text{ss}}_{g,n,d}$ are \[\mathfrak{h}_1,\cdots,\mathfrak{h}_l.\]
Let $ \mathfrak{h}_{i,k} $ be the proper transform of $\mathfrak{h}_i$ in $\mathfrak{U}_k$, where 
\[
k=\text{min}\{i\text{ }\vert\text{ the image of }\xi \text{ in } \mathfrak{U}_i \text{ lies in } \mathfrak{Z}_{(i)}\}.
\]
There exists a unique subset $\{i_1,\cdots,i_k\}\subset\{1,\cdots,l\}$ such that the image of $\xi$ in $\mathfrak{U}_k$ lies in the intersection of $\mathfrak{h}_{i_1,k},\cdots,\mathfrak{h}_{i_k,k}$.

\begin{definition}
With the notation above, we call $E_{i_1},\cdots, E_{i_k}$ the \textit{entangled tails} of $\xi$.
\end{definition}
Examples of entangled tails can be found in \cite[Section 2.3]{YZ}, specifically \cite[Example 2.3.1]{YZ}.

Next we define a gluing morphism which will be useful later. Fix $k\geq 1$ and let $ \mathfrak{M}^{\text{wt,ss}}_{g,n+k,d-kd_0}$ and $\mathfrak{M}^{\text{wt,ss}}_{0,1,d_0}$ be stacks of  $\epsilon_0$-semistable weighted curves. Note that we denote the pair $(0,d_0)$ as $d_0$ to lighten the notation and hereafter we will use $d_0=(0,d_0)$ wherever the context makes it clear. Define a gluing morphism
\[
\mathfrak{gl}_k: \mathfrak{M}^{\text{wt,ss}}_{g,n+k,d-kd_0}\times'(\mathfrak{M}^{\text{wt,ss}}_{0,1,d_0})^k\rightarrow \mathfrak{Z}_k\subset \mathfrak{M}^{\text{wt,ss}}_{g,n,d},
\]
where the product $\times'$ is formed by matching the sizes of automorphism groups at the last $k$ markings. These orders of automorphism groups at the $(n+i)$-th marked point define locally constant morphisms for each $i=1,\cdots,k$ denoted by 
\[
\mathfrak{r}_i:\mathfrak{M}^{\text{wt,ss}}_{g,n+k,d-kd_0}\times'(\mathfrak{M}^{\text{wt,ss}}_{0,1,d_0})^k\rightarrow \N.
\]
As $\mathfrak{M}^{\text{wt,ss}}_{g,n+k,d-kd_0}\times'(\mathfrak{M}^{\text{wt,ss}}_{0,1,d_0})^k$ is smooth, $\mathfrak{gl}_k$ factors through $\mathfrak{Z}_k^{\text{nor}}$, the normalization of $\mathfrak{Z}_k$. This induces a morphism 
\begin{equation}\label{spliteq}
\tilde{\mathfrak{gl}_k}: \tilde{\mathfrak{M}}_{g,n+k,d-kd_0}\times'(\mathfrak{M}^{\text{wt,ss}}_{0,1,d_0})^k\rightarrow \mathfrak{Z}_{(k)},
\end{equation}
which is \'etale of degree $k!/\prod_{i=1}^{i=k}\mathfrak{r}_i$ (\cite[Lemma 2.4.2]{YZ}).

\subsection{Calibration bundle}\label{calbun} 
Recall that $\mathfrak{Z}_1\subset\mathfrak{M}^{\text{wt},\text{ss}}_{g,n,d}$ is the reduced substack parametrizing curves with at least one degree $(0,d_0)$ rational tail.

\begin{definition}
	When $(g,n,d)\neq(0,1,d_0)$, the \textit{universal calibration bundle} is defined to be the line bundle $\O_{\mathfrak{M}^{\text{wt},\text{ss}}_{g,n,d}}(-\mathfrak{Z}_1)$; when $(g,n,d)=(0,1,d_0)$, the universal calibration bundle is the cotangent bundle at the unique marked point of rational tail.
\end{definition}

The calibration bundle of $\tilde{\mathfrak{M}}_{g,n,d}$, denoted by 
$$\mathbb{M}_{\tilde{\mathfrak{M}}_{g,n,d}},$$
is the pullback of the universal calibration bundle of $\mathfrak{M}^{\text{wt},\text{ss}}_{g,n,d}$ via the forgetfull map $\tilde{\mathfrak{M}}_{g,n,d}\rightarrow \mathfrak{M}^{\text{wt},\text{ss}}_{g,n,d}$. 

\begin{definition}\label{sscabun}
	The moduli stack of \textit{$\epsilon_0$-semistable curves with calibrated tails }is defined as 
	\[
	M\tilde{\mathfrak{M}}_{g,n,d}:=\P_{\tilde{\mathfrak{M}}_{g,n,d}}(\mathbb{M}_{\tilde{\mathfrak{M}}_{g,n,d}}\oplus\O_{\tilde{\mathfrak{M}}_{g,n,d}}).
	\]
\end{definition}

An $R$-point of $M\tilde{\mathfrak{M}}_{g,n,d}$ consists of 
\[
(\pi:C\rightarrow R,x,e,N,v_1,v_2),
\]
where
\begin{itemize}
\item $(\pi:C\rightarrow R, x,e)\in \tilde{\mathfrak{M}}_{g,n,d}(R)$;
\item $N$ is a line bundle on $R$;
\item $v_1\in\Gamma(R,\mathbb{M}_R\otimes N)$, $v_2\in \Gamma(R,N)$ have no common zeros, where $\mathbb{M}_R$ is the calibration bundle for the family of curves $\pi:C\rightarrow R$, defined by pulling back $\mathbb{M}_{\tilde{\mathfrak{M}}_{g,n,d}}$ via the map $R\to \tilde{\mathfrak{M}}_{g,n,d}$.
\end{itemize}

\subsection{Quasimap with entangled tails}\label{entails} 
Let $\tilde{V}, S, G $ be as defined in Section \ref{secTS}. Recall $X=[\tilde{V}/(S\times G)],X^{ss}=[\tilde{V}^{ss}/(S\times G)]$ and fix a curve class $\tilde{\beta}=(d',{\beta''})$ where $d'=(d'_1, \cdots, d'_r)$. Let $Q^{\text{pre}}_{g,n}(X^{ss},\tilde{\beta})$ be the stack of genus-$g$, $n$-marked quasimaps to $X^{ss}$ with curve class $\tilde{\beta}$. For each irreducible component $C_0'\subset C_0$, assign a pair $(d'_{C'_0},d''_{C'_0})$ given by the degree of class (i.e. $d'_{C'_0}:=\sum_{i=1}^{r}d'_i\vert_{C'_0}$ and $d''_{C'_0}:=\beta''_{C'_0}(L_{\theta})$). We call such a pair of non-negative numbers the degree of the curve and give it the dictionary order. 
\begin{remark}\label{rem3.6}
    Recall, $d'=(d'_1, \cdots, d'_r)$ is the degree of map $[u']:C_0\rightarrow [W^{ss}/S]$. We reuse the notation as the weight assigned to each irreducible component $C'_0$ given by $d'|_{C_0}:=\sum_{i=1}^{r} d'_i|_{C_0}$. 
\end{remark}
Define 
$$Q^{ss}_{g,n}(X^{ss},\tilde{\beta})\subset Q^{\text{pre}}_{g,n}(X^{ss},\tilde{\beta})$$ 
to be the open substack where the quasimaps have no 
\begin{itemize}
\item rational tails of degree $<(0,d_0)$ or $>(0,\infty)$ or rational bridges of degree $(0,0)$,
\item base points of length $>d_0$ on rational components (bridges and tails).
\end{itemize}

Let $\beta=(\beta',\tilde{\beta})$ be as before. Using degrees of quasimaps, we assign weights to irreducible components of the underlying curve. This defines a forgetful morphism $Q^{\text{ss}}_{g,n}(X^{ss},\tilde{\beta})\rightarrow \mathfrak{M}^{\text{wt,ss}}_{g,n,d}$ where $d=\text{deg}(\tilde{\beta})$. Define the moduli stack of $\epsilon_0$-semistable quasimaps to $[E^{ss}/G]$ as the fiber product
\[
{Q}^{\text{ss}}_{g,n}([E^{ss}/G],\beta):=\mathfrak{M}_{g,n}^{\text{tw,pre}}(Y,\beta')\times_{Q_{g,n}^{\text{pre}}([W^{ss}/S],d')}Q^{ss}_{g,n}(X^{ss},(d',{\beta''})).
\]
Composing with projection from above, we have
\[
{Q}^{\text{ss}}_{g,n}([E^{ss}/G],\beta)\rightarrow Q^{\text{ss}}_{g,n}(X^{ss},(d',{\beta''}))\rightarrow \mathfrak{M}^{\text{wt,ss}}_{g,n,d}.
\]

\begin{definition}\label{def17}
	We define the stack of \textit{$n$-pointed genus $g$, $\epsilon_0$-semistable quasimaps with entangled tails to $[E^{ss}/G]$ with curve class $\beta$} to be
	\[
	{Q}^{\sim}_{g,n}([E^{ss}/G], \beta):= {Q}^{\text{ss}}_{g,n}([E^{ss}/G],\beta)\times_{\mathfrak{M}^{\text{wt,ss}}_{g,n,d}} \tilde{\mathfrak{M}}_{g,n,d}.
	\]
\end{definition}

The construction above can be summarized in the following fibered diagram,
	\begin{equation}\label{cd9}
	\begin{tikzcd}
		{{Q}^{\sim}_{g,n}([E^{ss}/G], \beta)}\arrow[r]\arrow[d] & {Q^{\sim}_{g,n}(X^{ss},(d',{\beta''}))}\arrow[r] \arrow[d] & {\tilde{\mathfrak{M}}_{g,n,d}}\arrow[d]\\
		{Q_{g,n}^{\text{ss}}([E^{ss}/G],\beta)}\arrow[r]\arrow[d]& {Q_{g,n}^{\text{ss}}(X^{ss},(d',{\beta''}))}\arrow[d] \arrow[r]& {\mathfrak{M}^{\text{wt,ss}}_{g,n,d}}\\
		{\mathfrak{M}_{g,n}^{tw,\text{pre}}(Y,\beta')}\arrow[r] & {Q_{g,n}^{\text{pre}}([W^{ss}/S],d')}.
	\end{tikzcd}
\end{equation}
For a scheme $R$, an $R$-point of $Q^{\sim}_{g,n}([E^{ss}/G],\beta)$ is given by
\[
((\phi:C\rightarrow C_0,f,x),(C_0,x,[u],e)),
\]
where $((\phi:C\rightarrow C_0,f,x),(C_0,x,[u]))\in {Q}^{\text{ss}}_{g,n}([E^{ss}/G],\beta)(R)$ and $(C_0,x,e)\in \tilde{\mathfrak{M}}_{g,n,d}(R)$.

\begin{lemma}\label{lem16}
	The stack ${Q}^{\sim}_{g,n}([E^{ss}/G], \beta)$ is an Artin stack of finite type.
\end{lemma}

\begin{proof}
Recall from Theorem \ref{thm9}, $Q^{\text{pre}}_{g,n}([E^{ss}/G],\beta)$ is an Artin stack locally of finite type. Hence, so is ${Q}^{\text{ss}}_{g,n}([E^{ss}/G],\beta)$. As the morphism $\tilde{\mathfrak{M}}_{g,n,d}\rightarrow \mathfrak{M}^{\text{wt,ss}}_{g,n,d}$ is projective, by Definition \ref{def17} we get that ${Q}^{\sim}_{g,n}([E^{ss}/G], \beta)$ is an Artin stack of locally finite type.
	
	To show it is of finite type, it is enough to show that ${Q}^{\text{ss}}_{g,n}([E^{ss}/G],\beta)$ is bounded. We only need to show that the underlying curve has finitely many unmarked rational components, as the other components are bounded by $g,n$. Such components with restriction of $f$ being non-constant are bounded by $\beta'$. Hence it is enough to show that $f$-constant rational components with no marked points are finite. This in turn follows from the definition and noting that $Q_{g,n}^{\text{ss}}(X^{ss},(d',\tilde{\beta}))$ is bounded, see \cite[proof of Lemma 3.1.2]{YZ}.
\end{proof}

\begin{definition}
	An $R$-family of $\epsilon_0$-semistable quasimaps with entangled tails is called \textit{$\epsilon_+$-stable} if the underlying family of quasimaps to $[E^{ss}/G]$ is $\epsilon_+$-stable. We denote the moduli of $n$-marked, genus $g$, $\epsilon_+$-stable quasimaps of degree $\beta$ to $[E^{ss}/G]$ by ${\tilde{Q}}^{\epsilon_{+}}_{g,n}([E^{ss}/G], \beta)$.
\end{definition}

 From the definition, we have a natural isomorphism 
\[
{\tilde{Q}}^{\epsilon_{+}}_{g,n}([E^{ss}/G], \beta)\cong {Q}^{\epsilon_{+}}_{g,n}([E^{ss}/G],\beta)\times_{\mathfrak{M}^{\text{wt,ss}}_{g,n,d}} \tilde{\mathfrak{M}}_{g,n,d}.
\]
As ${Q}^{\epsilon_{+}}_{g,n}([E^{ss}/G],\beta)$ is a proper Deligne--Mumford stack, we have
\begin{lemma}
	The stack ${\tilde{Q}}^{\epsilon_{+}}_{g,n}([E^{ss}/G], \beta)$ is a proper Deligne--Mumford stack.
\end{lemma}

Similar to the obstruction theory of ${Q}^{\epsilon_{+}}_{g,n}([E^{ss}/G],\beta)$ in Section \ref{obthr}, we can construct a perfect obstruction theory for ${\tilde{Q}}^{\epsilon_{+}}_{g,n}([E^{ss}/G], \beta)$ relative to $\mathfrak{Bun}_{S}^{tw}$. In fact, the virtual classes induced by these two obstruction theories are related by the following lemma.

\begin{lemma}\label{lem21}
	Under the forgetful morphism ${\tilde{Q}}^{\epsilon_{+}}_{g,n}([E^{ss}/G], \beta)\rightarrow {Q}^{\epsilon_{+}}_{g,n}([E^{ss}/G],\beta)$, the pushforward of $[{\tilde{Q}}^{\epsilon_{+}}_{g,n}([E^{ss}/G], \beta)]^{\text{vir}}$ is equal to $[{Q}^{\epsilon_{+}}_{g,n}([E^{ss}/G],\beta)]^{\text{vir}}$.
\end{lemma}

\begin{proof}
This follows from \cite[Theorem 5.0.1]{Cos}, see also \cite{HW21}. 
\end{proof}

\subsubsection{Splitting off entangled tails} 
Let $$\mathfrak{C}_k^*\subset \tilde{\mathfrak{M}}_{g,n,d}$$ 
be the locus where there are exactly $k$ entangled tails. This is a locally closed smooth substack of codimension one.  Define the boundary divisor $$\mathcal{D}_k\subset \tilde{\mathfrak{M}}_{g,n,d}$$ as the closure of $\mathfrak{C}_k^*$. It is a smooth divisor with natural maps 
\[
\mathcal{D}_k\rightarrow\mathfrak{C}_k\rightarrow\mathfrak{Z}_{(k)}.
\]
where, $\mathfrak{Z}_{(k)}$ is the proper transform of substack parameterizing curves with at least $k$ rational tails of degree $(0,d_0)$ and $\mathfrak{C}_k$ is the exceptional divisor of blowup along $\mathfrak{Z}_{(k)}$.

Write $\tilde{\mathfrak{gl}}_k^*\mathcal{D}_k$ for the pullback along the gluing morphism (\ref{spliteq}).

Next we define the restriction $\tilde{Q}^{\epsilon_{+}}_{g,n}([E^{ss}/G],\beta)|_{\tilde{\mathfrak{gl}}_k^*\mathcal{D}_k}$ by the following fibered diagram,
\[
\begin{tikzcd}
	{\tilde{Q}^{\epsilon_{+}}_{g,n}([E^{ss}/G],\beta)|_{\tilde{\mathfrak{gl}}_k^*\mathcal{D}_k}}\arrow[r]\arrow[d]&{\tilde{Q}^{\epsilon_{+}}_{g,n}([E^{ss}/G],\beta)}\arrow[r]\arrow[d]&{Q^{\epsilon_+}_{g,n}([E^{ss}/G],\beta)}\arrow[d]\\
	{\tilde{\mathfrak{gl}}_k^*\mathcal{D}_k}\arrow[r,"i_{\mathcal{D}}"]&{\tilde{\mathfrak{M}}_{g,n,d}}\arrow[r]&{\mathfrak{M}_{g,n,d}^{\text{wt,ss}}}.
\end{tikzcd}
\]
Define $[\tilde{Q}^{\epsilon_{+}}_{g,n}([E^{ss}/G],\beta)|_{\tilde{\mathfrak{gl}}_k^*\mathcal{D}_k}]^{\text{vir}}:=i_{\mathcal{D}}^![\tilde{Q}^{\epsilon_{+}}_{g,n}([E^{ss}/G],\beta)]^{\text{vir}}$.

Moreover, the map $\tilde{\mathfrak{gl}}_k^*\mathcal{D}_k\rightarrow \mathfrak{M}_{g,n,d}^{\text{wt,ss}}$ factors as the composition 
\[
\tilde{\mathfrak{gl}}_k^*\mathcal{D}_k\rightarrow \mathfrak{M}_{g,n+k,d-kd_0}^{\text{wt,ss}}\times'(\mathfrak{M}_{0,1,d_0}^{\text{wt,ss}})^k\rightarrow \mathfrak{M}_{g,n,d}^{\text{wt,ss}}.
\]
Completing the analogous fibered diagram, we get 
\begin{equation}\label{diag3.3}
  \begin{tikzcd}
	{\tilde{Q}^{\epsilon_{+}}_{g,n}([E^{ss}/G],\beta)|_{\tilde{\mathfrak{gl}}_k^*\mathcal{D}_k}}\arrow[r,"p"]\arrow[d]&{\bigsqcup_{\overrightarrow{\beta}}Q_{\overrightarrow{\beta}}}\arrow[r]\arrow[d]&{Q^{\epsilon_+}_{g,n}([E^{ss}/G],\beta)}\arrow[d]\\
	{\tilde{\mathfrak{gl}}_k^*\mathcal{D}_k}\arrow[r]&{\mathfrak{M}_{g,n+k,d-kd_0}^{\text{wt,ss}}\times'(\mathfrak{M}_{0,1,d_0}^{\text{wt,ss}})^k}\arrow[r]&{\mathfrak{M}_{g,n,d}^{\text{wt,ss}}},
\end{tikzcd}  
\end{equation}
where $\overrightarrow{\beta}=(\beta_0,\beta_1,\cdots,\beta_k)$ runs through all $(k+1)$-tuples of effective curve classes such that deg$(\beta_i)=(0,d_0)$ for each $i=1,\cdots,k$ along with
\[
\beta=\beta_0+\cdots+\beta_k,
\]
 and we denote
\[
Q_{\overrightarrow{\beta}}:=\tilde{Q}^{\epsilon_{+}}_{g,n+k}([E^{ss}/G],\beta_0)\times_{(I_{\mu}[E^{ss}/G])^k}\prod_{i=1}^k{Q}^{\epsilon_{+}}_{0,1}([E^{ss}/G],\beta_i).
\]

\begin{lemma}\label{lem22}
	\begin{equation}
		\begin{split}
			[\tilde{Q}^{\epsilon_{+}}_{g,n}&([E^{ss}/G],\beta)|_{\tilde{\mathfrak{gl}}_k^*\mathcal{D}_k}]^{\text{vir}}=\\
			&p^*\left(\sum_{\overrightarrow{\beta}}\Delta^!_{(I_{\mu}[E^{ss}/G])^k}[\tilde{Q}^{\epsilon_{+}}_{g,n+k}([E^{ss}/G],\beta_0)]^{\text{vir}}\boxtimes\prod_{i=1}^k[{Q}^{\epsilon_{+}}_{0,1}([E^{ss}/G],\beta_i)]^{\text{vir}}\right),
		\end{split}
	\end{equation}
where $\Delta_{(I_{\mu}[E^{ss}/G])^k}:(I_{\mu}[E^{ss}/G])^k\rightarrow (I_{\mu}[E^{ss}/G])^k\times (I_{\mu}[E^{ss}/G])^k$ is the diagonal morphism.
\end{lemma}

\begin{proof}
	This follows by an argument similar to the one for \cite[Lemma 3.2.1]{YZ}.
\end{proof}
\subsection{Definition of the master space}\label{subsec:defn_master_space} 
Let $g,n,d,d_0=1/\epsilon_0$ be as above. As in Section \ref{calbun}, let $\mathbb{M}_{\tilde{\mathfrak{M}}_{g,n,d}}$ be the calibration bundle of $\tilde{\mathfrak{M}}_{g,n,d} $ and let $M\tilde{\mathfrak{M}}_{g,n,d}$ be the moduli of curves with calibrated tails. 

\begin{definition}\label{def21}
	The \textit{moduli stack of genus $g$, $n$-marked $\epsilon_0$-semistable quasimaps with calibrated tails to $[E^{ss}/G]$ of curve class $\beta$} is defined to be
	\[
	MQ^{\sim}_{g,n}([E^{ss}/G], \beta):=Q^{\sim}_{g,n}([E^{ss}/G],\beta)\times_{\tilde{\mathfrak{M}}_{g,n,d}}M\tilde{\mathfrak{M}}_{g,n,d}. 
	\]
\end{definition}
For a scheme $R$, an $R$-point of $MQ^{\sim}_{g,n}([E^{ss}/G], \beta)$ can be expressed as 
	\[
	((\phi:C\rightarrow C_0,f,x),(C_0,x,[u],e,N,v_1,v_2)), 
	\]
	where $((\phi:C\rightarrow C_0,f,x),(C_0,x,[u],e))\in Q^{\sim}_{g,n}([E^{ss}/G],\beta)(R)$ and $(C_0,x,e,N,v_1,v_2)\in M\tilde{\mathfrak{M}}_{g,n,d}(R)$.
	
From diagram (\ref{cd9}), we have the following natural correspondence
	\begin{equation}\label{eq10}
     	MQ^{\sim}_{g,n}([E^{ss}/G], \beta)\cong \mathfrak{M}_{g,n}^{tw,\text{pre}}(Y,\beta')\times_{Q_{g,k}^{\text{pre}}([W^{ss}/S],d')}MQ^{\sim}_{g,n}(X^{ss},\tilde{\beta}),
	\end{equation}
where 
\[
MQ^{\sim}_{g,n}(X^{ss},\tilde{\beta}):=Q^{\sim}_{g,n}(X^{ss},\tilde{\beta})\times_{\tilde{\mathfrak{M}}_{g,n,d}}M\tilde{\mathfrak{M}}_{g,n,d}.
\]

Let $((\phi:C\rightarrow C_0,f,x),(C_0,x,[u],e))\in Q^{\sim}_{g,n}([E^{ss}/G],\beta)(\C)$ be a geometric point of $\epsilon_0$-semistable quasimaps with entangled tails. Recall from the definition in Section \ref{entails} that the curve $C_0$ may have a rational tail of degree $(0,d_0)$ and base points of length $d_0$.

\begin{definition}\label{rbp}
We call a base point a \textit{relevant base point} if it is contained in a rational component.
\end{definition}
	
\begin{definition}
A degree-$(0,d_0)$ rational tail $E\subset C_0$ is called a \textit{constant tail }if $E$ contains a (relevant) base point of length $d_0$. 
\end{definition}
	
We now describe the stability condition.

\begin{definition}\label{scal}
	An $R$-family of $\epsilon_0$-semistable quasimaps with calibrated tails 
	\[
	((\phi:C\rightarrow C_0,f,x),(C_0,x,[u],e,N,v_1,v_2))
	\]
	is \textit{$\epsilon_0$-stable} if over every geometric point $r$ of $R$,
	\begin{enumerate}
		\item any constant tail is an entangled tail;
		\item if the geometric fiber $C_0(r)$ of $C_0$ contains rational tails of degree $(0,d_0)$, then all the length $d_0$ relevant base points (in the sense of Definition \ref{rbp}) are contained in these tails;
		\item if $v_1(r)=0$, then $((\phi:C\rightarrow C_0,f,x),(C_0,x,[u]))(r)$ is an $\epsilon_{+}$-stable quasimap to $[E^{ss}/G]$;
		\item if $v_2(r)=0$, then $((\phi:C\rightarrow C_0,f,x),(C_0,x,[u]))(r)$ is an $\epsilon_{-}$-stable quasimap to $[E^{ss}/G]$.
	\end{enumerate}
	
We denote by $$MQ^{\epsilon_0}_{g,n}([E^{ss}/G],\beta)$$ the category of \textit{genus-$g$, $n$-marked, $\epsilon_0$-stable quasimaps with calibrated tails to $[E^{ss}/G]$ of curve class $\beta$}.
\end{definition}

\begin{lemma}\label{lem25}
The stability condition in Definition \ref{scal} is an open condition.
\end{lemma}
\begin{proof}
Conditions (3) and (4) are open conditions since $\epsilon$-stability is an open condition in the space of prestable quasimaps. Conditions (1) and (2) are open: notice that they are open conditions in $MQ^{\sim}_{g,n}(X^{ss},\tilde{\beta})$ by \cite[Lemma 4.1.4]{YZ} and hence open in $MQ^{\sim}_{g,n}([E^{ss}/G],{\beta})$ by equation (\ref{eq10}). 
\end{proof}

Let 
\[
\xi =((\phi:C\rightarrow C_0,f,x),(C_0,x,[u],e,N,v_1,v_2))\in MQ^{\sim}_{g,n}([E^{ss}/G],\beta)(\C)
\]
be an $\epsilon_0$-semistable quasimap with calibrated tails. Let $E\subset C_0$ be a degree-$(0,d_0)$ rational tail with $y\in E$ as its unique node (or marking).
\begin{definition}
	$E$ is said to be a \textit{fixed tail} if $\text{Aut}(E,y,u\vert_{E})$ is infinite.
\end{definition}

\begin{lemma}\label{infaut}
	The family of quasimaps to  $[E^{ss}/G]$ with entangled tails
	\[
	\eta=((\phi:C\rightarrow C_0,f,x),(C_0,x,[u],e))
	\]
	obtained from $\xi$ has infinitely many automorphisms if and only if
	\begin{enumerate}
		\item there is at least one degree $(0,d_0)$ rational tail, and
		
		\item each entangled tail is a fixed tail.
	\end{enumerate}
Moreover, when $\text{Aut}(\eta)$ is infinite, its identity component $\Gamma\subset \text{Aut}(\eta)$ is isomorphic to $\C^*$ and acts on each $T_{y_i}E_i$ by the same nonzero weight for $i=1,\cdots,k$.
\end{lemma}

\begin{proof}
First we show the ``only if'' part. Consider the union of $f$-nonconstant irreducible components  $C''\subset C$. By prestability condition of quasimaps,  for any irreducible rational component $C'\subset C''$ in this union ($f|_{C'}$ is non-constant), $C'$ gets contracted under the contraction map $\phi$. Hence, $(\phi(C''),x,[u]|_{C''})$ is a $0^+$-stable quasimap to $X$ and therefore has finitely many automorphisms. This means $\eta|_{C'}$ has infinitely many automorphism only if $f|_{C'}$ is constant. 
	
So we can assume that $f$ is constant on the underlying curve $C$. In particular, the contraction map $\phi$ is an isomorphism and $\eta$ has infinitely many automorphisms only if $(C,x,[u],e)$ has infinitely many automorphisms. Then we get ``only if'' implication  by the ``only if" part of \cite[Lemma 4.1.8]{YZ}.
	
The ``if'' part and rest of the lemma follows directly from the ``if'' part of \cite[Lemma 4.1.8]{YZ}.
\end{proof}

\begin{lemma}\label{lem28}
If $\xi$ is $\epsilon_0$-stable, then $\xi$ has finitely many automorphisms.
\end{lemma}
\begin{proof}
This follows from Lemma \ref{infaut} and an argument similar to \cite[Lemma 4.1.10]{YZ}.
\end{proof}

\begin{Prop}
$MQ^{\epsilon_0}_{g,n}([E^{ss}/G],\beta)$ is a Deligne--Mumford stack of finite type over $\C$.
\end{Prop}

\begin{proof}
	By Lemma \ref{lem16}, $Q^{\sim}_{g,n}([E^{ss}/G],\beta)$ is an Artin stack of finite type. By Definition \ref{def21}, so is $MQ^{\sim}_{g,n}([E^{ss}/G],\beta)$. Finally by Lemmas \ref{lem25} and \ref{lem28}, $MQ^{\epsilon_{0}}_{g,n}([E^{ss}/G],\beta)$ is a Deligne--Mumford stack of finite type.
\end{proof}

\subsubsection{Virtual fundamental class}\label{obthrmaster}
Similar to Section \ref{obthr}, we write the master space as fiber product of two spaces who have relative obstruction theory over a common space (in this case, the common space is $\mathfrak{Bun}^{\text{tw,wt}}_S$) and then combine these two obstruction theory to get an obstruction theory for the master space. 
We define a moduli stack $$\mathcal{G}^{\text{tw},d'}_{g,n,{\beta''}}$$ 
parameterizing objects of the form 
$(C^{\text{wt}}_0,x,u'')$ where $C^{\text{wt}}_0$ is a weighted twisted curve such that all degrees add up to $d'$ and $(C_0,x,u'')$ is an $n$-marked genus $g$, quasimap to $[V^{ss}/(S\times G)]$ of degree ${\beta''}$. Here $C_0$ is the underlying twisted curve without the weights. 

The map $Q_{g,n}^{\text{ss}}(X^{ss},(d',{\beta''}))\rightarrow\mathfrak{M}^{\text{wt,ss}}_{g,n,d}$ factors through $\mathcal{G}^{\text{tw},d'}_{g,n,{\beta''}}$. Moreover, we have a fibered diagram 
\[
\begin{tikzcd}
	{Q_{g,n}^{\text{ss}}(X^{ss},(d',{\beta''}))}\arrow[r]\arrow[d]&{\mathcal{G}^{\text{tw},d'}_{g,n,{\beta''}}}\arrow[d]\arrow[r]&{\mathfrak{M}^{\text{wt,ss}}_{g,n,d}}\\
	{Q_{g,n}^{\text{pre}}([W^{ss}/S],d')}\arrow[r]&{\mathfrak{Bun}^{\text{tw,wt}}_S}.
\end{tikzcd}
\]
Here $\mathfrak{Bun}^{\text{tw,wt}}_S:=\mathfrak{Bun}^{\text{tw}}_S\times_{\mathfrak{M}^{\text{ss}}_{g,n,d}}\mathfrak{M}^{\text{wt,ss}}_{g,n,d'}$ is a stack of principal $S$-bundles over weighted twisted curves.  $Q_{g,n}^{\text{ss}}(X^{ss},(d',{\beta''}))\rightarrow\mathcal{G}^{\text{tw},d'}_{g,n,{\beta''}}$ is obtained by forgetting the data of quasimap to $[W^{ss}/S]$ along with assigning degree of its class on first coordinate as weight. In particular, $((C_0,x),[u])\mapsto ((C_0^{wt},x),u'')$ where $u''$ is composition of $C_0 \xrightarrow{[u]} X^{ss} \rightarrow [V^{ss}/(S\times G)]$ of class $\beta''$. Moreover, we also assign a weight on irreducible component given by class $d'$ (See Remark \ref{rem3.6}). Finally, $\mathcal{G}^{\text{tw},d'}_{g,n,{\beta''}}\rightarrow \mathfrak{M}^{\text{wt,ss}}_{g,n,d}$ is given by forgetting the morphisms and changing the weights on each irreducible curve $C$ from $d'\mapsto (d',\beta_{C}''(L_{\theta}))$.

Define 
\[
\tilde{M\mathcal{G}}^{\text{tw},d'}_{g,n,{\beta''}}:=\mathcal{G}^{\text{tw},d'}_{g,n,{\beta''}}\times_{\mathfrak{M}^{\text{wt,ss}}_{g,n,d}}M\tilde{\mathfrak{M}}_{g,n,d}.
\]
Then by diagram (\ref{cd9}) and the definition of $MQ^{\sim}_{g,n}([E^{ss}/G],\beta)$, we have
\[
MQ^{\sim}_{g,n}([E^{ss}/G],\beta)=\mathfrak{M}^{\text{tw,pre}}_{g,n}(Y,\beta')\times_{\mathfrak{Bun}^{\text{tw,wt}}_S} \tilde{M\mathcal{G}}^{\text{tw},d'}_{g,n,{\beta''}}.
\]

Let $\pi:\mathfrak{C_0}\rightarrow \tilde{M\mathcal{G}}^{\text{tw},d'}_{g,n,{\beta''}}$ be the universal curve, and let $u:\mathfrak{C_0}\rightarrow[V/(S\times G)]$ be the universal map. Similar to \cite[Section 4.2]{YZ}, the forgetful morphism $\tilde{M\mathcal{G}}^{\text{tw},d'}_{g,n,{\beta''}}\rightarrow M\tilde{\mathfrak{M}}_{g,n,d}$ admits a relative perfect obstruction theory 
\[
(R\pi_*(u^*\mathbb{T}_{[V/(S\times G)]}))^{\vee}\rightarrow \mathbb{L}_{\tilde{M\mathcal{G}}^{\text{tw},d'}_{g,n,{\beta''}}/M\tilde{\mathfrak{M}}_{g,n,d}}.
\]
Let $\mathbb{E}_{\tilde{M\mathcal{G}}^{\text{tw},d'}_{g,n,{\beta''}}}$ be the absolute perfect obstruction theory on $\tilde{M\mathcal{G}}^{\text{tw},d'}_{g,n,{\beta''}}$. Now as $\mathfrak{Bun}^{\text{tw}}_S$ is a smooth Artin stack, we have a perfect obstruction theory $\mathbb{E}_{\mu_1}$ relative to the forgetful morphism 
\[
\mu_1:\tilde{M\mathcal{G}}^{\text{tw},d'}_{g,n,{\beta''}}\rightarrow \mathfrak{Bun}^{\text{tw}}_S,
\]
defined as 
\[
\mathbb{E}_{\mu_1}:=\text{Cone}(\mu_1^*\mathbb{L}_{\mathfrak{Bun}^{\text{tw}}_S}\rightarrow \mathbb{E}_{\tilde{M\mathcal{G}}^{\text{tw},d'}_{g,n,{\beta''}}}).
\]
As it is defined by forgetting the weights, $\mathfrak{Bun}^{\text{tw,wt}}_S\rightarrow \mathfrak{Bun}^{\text{tw}}_S$ is \'etale, and we obtain a perfect obstruction theory relative to $\mathfrak{Bun}^{\text{tw,wt}}_S$ which we again denote by $\mathbb{E}_{\mu_1}$.

Let $\mathbb{E}_{\mu_2}$ be the perfect relative obstruction theory of $\mathfrak{M}^{\text{tw,pre}}_{g,n}(Y,\beta')$ relative to 
\[
\mu_2:\mathfrak{M}^{\text{tw,pre}}_{g,n}(Y,\beta')\rightarrow \mathfrak{Bun}^{\text{tw,wt}}_S.
\]
Note that we get a perfect obstruction theory relative to $\mathfrak{Bun}^{\text{tw}}_S$ from Section \ref{obthr}, which again can be lifted to $\mathfrak{Bun}^{\text{tw,wt}}_S$.

We define (suppressing obvious pullbacks)
\[
\mathbb{E}_{\mu}:=(\mathbb{E}_{\mu_1}\oplus\mathbb{E}_{\mu_2})|_{MQ^{\epsilon_0}_{g,n}([E^{ss}/G],\beta)}\in D^b(MQ^{\epsilon_0}_{g,n}([E^{ss}/G],\beta))
\]
as the perfect relative obstruction theory of $MQ^{\epsilon_0}_{g,n}([E^{ss}/G],\beta)$ over $\mathfrak{Bun}^{\text{tw,wt}}_S$. 
This defines a virtual fundamental class 
\[
[MQ^{\epsilon_0}_{g,n}([E^{ss}/G],\beta)]^{\text{vir}}\in A_*(MQ^{\epsilon_0}_{g,n}([E^{ss}/G],\beta)).
\]

From the relative obstruction theory, we obtain an absolute one defined as 
\[
\mathbb{E}_{MQ}:=\text{Cone}(\mathbb{E}_{\mu}(-1)\rightarrow\mathbb{L}_{\mu}(-1)\rightarrow\mu^*\mathbb{L}_{\mathfrak{Bun}^{\text{tw,wt}}_S}),
\]
see \cite[Remark 4.5.3]{ckm2014stable}.

\subsection{Properness of master space}\label{subsec:proper_master_space}
The purpose of this Subsection is to prove the following
\begin{Prop}
	$MQ^{\epsilon_0}_{g,n}([E^{ss}/G],\beta)$ is proper over $\C$.
\end{Prop}
We will use valuative criteria for properness. Let $(R,\bullet)$ be a complete discrete valuation $\C$-algebra with $K$ as its fraction field and residue field $\C$. We will prove that given any 
	\[
	\xi^*= ((\phi^*:C^*\rightarrow C_0^*,f^*,x^*),(C_0^*,x^*,[u^*],e^*,N^*,v^*_1,v^*_2))\in MQ^{\epsilon_0}_{g,n}([E^{ss}/G],\beta)(K),
	\]
	it has a unique (up to finite base change) extension to
	\[
	\xi=	((\phi:C\rightarrow C_0,f,x),(C_0,x,[u],e,N,v_1,v_2))\in MQ^{\epsilon_0}_{g,n}([E^{ss}/G],\beta)(R).
	\]
\subsubsection{Case 1:} Assume $(g,n,d)\neq (0,1,d_0)$ and $\xi^*$ does not have degree-$(0,d_0)$ rational tails. Then 
	\[
	\eta^*= ((\phi^*:C^*\rightarrow C^*_0,f^*,x^*),(C^*_0,x^*,[u^*]))
	\]
	is a $K$-family of $\epsilon_{-}$-stable quasimaps to $[E^{ss}/G]$. As ${Q}^{\epsilon_{-}}_{g,n}([E^{ss}/G],\beta)$ is proper, it uniquely extends to an $R$-family of $\epsilon_{-}$-stable quasimaps
	\[
	\eta_-=((\phi_-:C_-\rightarrow C_{0-},f_-,x_-),((C_{0-},x_-,[u_-])).
	\]
	As  the representable morphism $M\tilde{\mathfrak{M}}_{g,n,d}\rightarrow \tilde{\mathfrak{M}}_{g,n,d}\rightarrow \mathfrak{M}^{\text{wt,ss}}_{g,n,d}$ is proper, we have an $\epsilon_0$-semistable extension of $\xi^*$
	\[
	\xi_-= ((\phi_-:C_-\rightarrow C_{0-},f_-,x_-),(C_{0-},x_-,[u_-],e_-,N_-,v_{1-},v_{2-})).
	\]
We need to show that it is $\epsilon_0$-stable. 
	
First we start with a definition. Consider a totally ramified finite base change $\text{Spec}R'\rightarrow \text{Spec}R$ of degree $r$. Let $K'$ be the fraction field of $R'$.
	
\begin{definition}
A modification of $\eta_-$ of degree $r$ is the family of $\epsilon_0$-semistable quasimaps over $R'$
		\[
		\tilde{\eta}=((\tilde{\phi}:\tilde{C}\rightarrow\tilde{C}_0,\tilde{f},\tilde{x}),(\tilde{C}_0,\tilde{x},[\tilde{u}])),
		\]
together with an isomorphism 
		\[
		\tilde{\eta}\vert_{\text{Spec}K'}\cong\eta_-\vert_{\text{Spec}K'}.
		\]
\end{definition}
	
Let $\eta'=((\phi':C'\rightarrow C'_0,f',x'),(C'_0,x',[u']))$ be the pullback of $\eta_-$ to $\text{Spec}R'$.
	
\begin{lemma}
		The family of curves $C'_0$ is obtained from $\tilde{C}_0$ by contracting the rational tails of degree $(0,d_0)$ on the special fiber.
\end{lemma}
	
\begin{proof}
		Using separatedness of $Q^{\epsilon_{-}}_{g,n}([E^{ss}/G],\beta)$, this follows by a similar argument given in \cite[Lemma 5.2.2]{YZ} . 
\end{proof}
	
	As the extension $R\subset R'$ is totally ramified, $C'_0(\bullet)\cong C_{0-}(\bullet)$. Let $p_1,\cdots, p_k$ be the length $d_0$ relevant base points in the special fiber. For the modifications to be $\epsilon_0$-stable, rational tails of degree $(0,d_0)$ should contract to length $d_0$ relevant base points on $C'_0(\bullet)$ (Condition (2) of Definition \ref{scal}). Hence, over each $p_i$ we have a contracted tail $E_i$ of degree $(0,d_0)$. Let each $E_i$ intersect the other component of $\tilde{C}_0(\bullet)$ at some $A_{a_1-1}$-singularity. Such a modification $\tilde{\eta}$ is said to have singularity type $(a_1/r,\cdots, a_k/r)$.
	
	Singularity types of all such modifications are bounded from above by $(b_1,\cdots,b_k)\in \Q_{>0}\cup\{\infty\}$ which we describe briefly. Let $B^*\subset C_{0-}(K)$ be the locus of length $d_0$ relevant base points and let $B\subset C_{0-}$ be its closure. Let $p_{l+1},\cdots, p_k$ be the length $d_0$ relevant base points in the special fiber contained in $B$. Assign $b_i=\infty$ for $i={l+1},\cdots,k$. Replace $B$ by additional orbifold markings \cite[Lemma 5.2.4]{YZ} so that the generic fiber of $\eta^*$ is an $\epsilon_{+}$-stable quasimap. Its $\epsilon_{+}$-extension to $R$ will have $l$ degree-$(0,d_0)$ rational tails $E_1,\cdots, E_l$ corresponding to $p_1,\cdots, p_l$. We define $(b_1,\cdots, b_l)$ to be the singularity type of this $\epsilon_{+}$-extension. 
	
	To see the bound, we record \cite[Lemma 5.2.5]{YZ}. Even though the lemma was proven for quasimaps to $[V^{ss}/G]$ with affine $V$, the same proof works for quasimaps to $[E^{ss}/G]$.
	
\begin{lemma}[c.f. \cite{YZ}, Lemma 5.2.5]\label{unimod}
For $a:=(a_1,\cdots, a_k)\in \Q_{>0}^k$ and sufficiently divisible $r$, the following are  equivalent:
		\begin{enumerate}
			\item $\eta_-$ has a modification $\tilde{\eta}$ of degree $r$ and singularity type $a$;
			\item $a_i\leq b_i$ for each $i$. 
		\end{enumerate}
		Moreover, assuming $\tilde{\eta}$ exists, then
		\begin{itemize}
			\item $\tilde{\eta}$ is uniquely determined by $a$ and $r$.
			\item for each $i$, $E_i$ contains no length-$d_0$ relevant base point if and only if $a_i=b_i$, where $E_i$ is a rational tail of $\tilde{\eta}$ lying over $p_i$.
		\end{itemize}
\end{lemma}
	
Recall from the definition of $\epsilon_0$-semistable calibrated bundle (Definition \ref{sscabun}), $v_1$ and $v_2$ have no common zeros. Denote their vanishing order at the closed point by  $\text{ord}(v_1)$ and $\text{ord}(v_2)$. When $v_{1-}(\bullet)\neq0$, the stability condition (see Definition \ref{scal}) is trivially true (as $\eta_-$ is $\epsilon_{-}$-stable quasimap, it does not have degree $d_0$ rational tails), and we can take $\xi:=\xi_-$. To show that this is the unique stable extension, consider a degree $r$ modification $\tilde{\eta}$ with singularity type $(a_1,\cdots,a_k)$. If $a_i>0$ for some $i$, it means that there is a degree-$(0,d_0)$-rational tail in the special fiber. Moreover, applying \cite[Lemma 2.10.1]{YZ} to extensions $C_-$ and $\tilde{C}$, we have the following relation,
\begin{equation}\label{ordrel}
\text{ord}(\tilde{v}_1)-\text{ord}(\tilde{v}_2)=r(\text{ord}(v_{1-})-\text{ord}(v_{2-}))-r\sum_{i=1}^{k}a_i.
\end{equation}
As $\text{ord}(v_{1-})=0$, we have $0\leq \text{ord}(\tilde{v}_1)<\text{ord}(\tilde{v}_2)$. Hence $\tilde{v}_2(\bullet)=0$ but also contains degree $(0,d_0)$-rational tail, contradicting stability. Similarly, if there are no length-$d_0$ base points in the special fiber, $\xi_-$ is $\epsilon_0$-stable and we are done. 
	
So let us assume $v_{1-}(\bullet)=0$ and let $p_1,\cdots,p_k\in C_{0-}({\bullet})$ be length-$d_0$ relevant base points in the special fiber. Let $\tilde{\eta}$ be a non-trivial modification of $\eta_-$ with singularity type $(a_1,\cdots,a_k)$. Put $\delta:=\text{ord}(v_{1-})-\text{ord}(v_{2-})$ and $|a|:=\sum_{i=1}^{k}(a_i)$.  By Lemma \ref{unimod}, such a modification is stable if and only if 
	\begin{equation}
		\begin{cases}
			& |a|\leq\delta;\\
			& 0<a_i\leq b_i \text{, for all } i=1,\cdots,k;\\
			& \text{if } |a|<\delta \text{ then } a_i=b_i \text{ for all } i=1,\cdots,k;\\
			&\text{if } a_i<b_i \text{ then } a_i \text{ is maximal among } a_i,\cdots, a_k, \text{ for all } i=1,\cdots, k.
		\end{cases}       
	\end{equation}
There exists a unique solution of $(a_1,\cdots, a_k)$ up to a totally ramified base change. By Lemma \ref{unimod}, $(a_1,\cdots, a_k)$ and $r$ completely determine the modification $\tilde{\eta}$ which in turn gives a unique $\epsilon_0$-stable extension $\tilde{\xi}$. We take $\xi=\tilde{\xi}$ as the required unique (up to base change) extension.

\subsubsection{Case 2:} We assume $(g,n,d)=(0,1,d_0)$. Let $\beta= (0,\beta_0)$, so that $d_0=\beta_0(L_{\theta})$. Recall that the curve $C_0$  cannot have rational tails with degrees less than $d_0$. Hence, $C_0$ is an irreducible rational tail with one special point. The map $\phi:C\rightarrow C_0$ is an isomorphism and $f:C\rightarrow Y$ is a constant map.

In this case,  $Q^{\sim}_{0,1}([E^{ss}/G],\beta)\cong Y\times_{[W^{ss}/S]}Q^{\sim}_{0,1}(X^{ss},(0,\beta_0))$. This implies the following isomorphism for $\epsilon_0$-semistable quasimaps with calibrated tails, 
\[
MQ^{\sim}_{0,1}([E^{ss}/G],\beta)\cong Y\times_{[W^{ss}/S]}MQ^{\sim}_{0,1}(X^{ss},(0,\beta_0)).
\]
Now properness follows from that of $MQ^{\epsilon_0}_{0,1}(X^{ss},(0,\beta_0))$ \cite[Section 5.3]{YZ}.

\subsubsection{Case 3:} Assume that the generic fiber contains degree-$(0,d_0)$ rational tails and let $\mathcal{E}_1^*,\cdots, \mathcal{E}_l^*\subset C_0^*$ be the entangled rational tails. Let $C^*_{0\sim}$ be the union of other components along with non-entangled rational tails. Viewing nodes on $C^*_{0\sim}$ as new markings $y^*_{\sim}$, we obtain a family of quasimaps over $\text{Spec}K$
\[
\eta^*_\sim=((\phi^*_{\sim}:C^*_{\sim}\rightarrow C^*_{0\sim},f^*_\sim=f^*\vert_{C^*_{0\sim}},x^*_\sim,y^*_\sim),(C^*_{0\sim},[u^*_{\sim}]=[u^*]\vert_{C_{0\sim}^{*}},x^*_\sim,y^*_\sim)).
\]
Similarly, consider the node on entangled tails as marking $z^*_i$ to obtain quasimaps of type $(g,n,\beta)=(0,1,(0,d_0))$ for $i=1,\cdots,l$,
\[
\eta^*_i=((\phi^*_i:\mathcal{E}_i^*\cong \mathcal{E}^*_i,f^*_i, z^*_i),(\mathcal{E}^*_{i},[u^*_i],z^*_i)).
\]
$\eta^*_\sim$ does not contain any entangled tails. Hence by stability condition (Definition \ref{scal} conditions (1)-(2)), there are no constant tails and no length $d_0$ base points. Hence $\eta^*_\sim$ is $\epsilon_{+}$-stable. In general, any stable extension can be obtained by gluing $\epsilon_{+}$-stable $\eta_\sim$ and $\epsilon_0$-semistable $\eta_i$ of type $(g,n,\beta)=(0,1,(0,d_0))$ for $i=1,\cdots,l$ at markings $z_i,y_i$.

Let $\eta_\sim$ and $\eta_i$ be $\epsilon_0$-stable extensions (if they exist) of $\eta_\sim^*$ and $\eta^*_i$ respectively. By properness of $Q^{\epsilon_+}_{g,n+l}([E^{ss}/G],\beta-ld_0)$, we can take $\eta_\sim$ to be the $\epsilon_+$-extension of $\eta_\sim^*$. However $\eta^*_i$ is only $\epsilon_0$-semistable and does not have a unique extension. Still, there is a unique choice of $\eta_1,\cdots,\eta_l$ such that $\xi$ is stable. To see this, note that $f^*_i$ is a constant morphism. Hence we can consider each $\eta_i^*$ to be an element of $Q^{\text{ss}}_{0,1}(X^{ss},(0,d_0))$. Therefore, uniqueness follows by an argument similar to the one in \cite[Section 5.4]{YZ}.


\section{Localization on Master Space}\label{sec:localization}
Consider the $\C^*$-action on $MQ^{\epsilon_0}_{g,n}([E^{ss}/G],\beta)$ given by scaling the section $v_1$: for $\lambda\in \C^*$,
\begin{equation*}
	\begin{split}
		&\lambda\cdot((\phi:C\rightarrow C_0,f,x),(C_0, x, [u], e,N,v_1,v_2))\\  
        := &((\phi:C\rightarrow C_0,f,x),(C_0,x,[u],e,N,\lambda v_1,v_2)).
\end{split}
\end{equation*}

\subsection{Fixed components}
We describe $\C^*$-fixed components of $MQ^{\epsilon_0}_{g,n}([E^{ss}/G],\beta)$.
\subsubsection{$v_1=0$}\label{fixcomp1}
Let $$F^+\subset MQ^{\epsilon_0}_{g,n}([E^{ss}/G],\beta)$$ denote the fixed component of $MQ^{\epsilon_0}_{g,n}([E^{ss}/G],\beta)$ defined by $v_1=0$. By stability condition (condition (3) of Definition \ref{scal}) and forgetting the trivial data $N,v_1,v_2$, we have 
\begin{align}
	F^+\cong \tilde{Q}^{\epsilon_+}_{g,n}([E^{ss}/G],\beta), \text{ under which } [F^+]^{\text{vir}}= [\tilde{Q}^{\epsilon_+}_{g,n}([E^{ss}/G],\beta)]^{\text{vir}}.
\end{align}
The identification of virtual fundamental classes follows from the fact that $\{v_1=0\}\subset M\tilde{\mathfrak{M}}_{g,n,d}$ is isomorphic to $\tilde{\mathfrak{M}}_{g,n,d}$. This also implies that the calibration bundle is the virtual normal bundle with a $\C^*$-action of weight $1$.

\subsubsection{$v_2=0$}\label{fixcomp2}
Similarly when $v_2=0$, we have a fixed component denoted by 
$$F^-\subset MQ^{\epsilon_0}_{g,n}([E^{ss}/G],\beta).$$ 
By the stability condition (condition (4) of Definition \ref{scal}), the underlying quasimap is $\epsilon_{-}$-stable and hence does not have degree-$(0,d_0)$ rational tails. Hence, in this case we have no entangled tails and $\{v_2=0\}\subset M\tilde{\mathfrak{M}}_{g,n,d}$ is isomorphic to $\mathfrak{M}^{\text{wt,ss}}_{g,n,d}$. Therefore, we have 
\begin{align}
	F^-\cong {Q}^{\epsilon_-}_{g,n}([E^{ss}/G],\beta), \text{ under which } [F^-]^{\text{vir}}= [{Q}^{\epsilon_-}_{g,n}([E^{ss}/G],\beta)]^{\text{vir}}.
\end{align}
Note that for $g=0,n=1$ and deg$(\beta)=(0,d_0)$, $Q^{\epsilon_-}_{g,n}([E^{ss}/G],\beta)$ is empty and $v_2\neq 0$. When it is nonempty, the virtual normal bundle is the dual of the calibration bundle with $\C^*$-action of weight $(-1)$, since we look at 
$\{v_2=0\}\subset M\tilde{\mathfrak{M}}_{g,n,d}$.

\subsubsection{$v_1, v_2\neq 0$, Case I}
 When $g=0,n=1,\text{deg}(\beta)=(0,d_0)$, the curve is irreducible. The $\C^*$-fixed component $F_{\beta}$ is given by
\[
F_{\beta}=\{\xi\vert\text{ the domain curve is a single fixed tail, }v_1\neq0,v_2\neq0\}.
\]
Let $C$ be the domain curve with the unique marked point $x_{\star}$. As $\text{deg}(\beta)=(0,d_0)$, $f:C\rightarrow Y$ is a constant map (which means that the contraction map is an isomorphism), and $[u]:C\rightarrow X$ is a genus $0$, $1$-pointed quasimap to $X^{ss}$ with degree $\beta_0$ such that deg$(\beta_0)=(0,d_0)$. Let $\mathfrak{r}_{\star}$ be the restriction of $\mathfrak{r}$ at the unique point $x_{\star}$. 

Fix a nonzero tangent vector $v_{\infty}$ at $\infty\in \P^1$. There is a unique isomorphism $C\rightarrow \P^1$ mapping the marked point to $\infty$, the base point to $0$, and sending $(v_2/v_1)^{\otimes\mathfrak{r}_{\star}}$ to $v_{\infty}$. This together with the morphism $[u]$ determine a point in quasimap graph space $QG_{0,1}(X^{ss},\beta_0)$ (whose definition is recalled in Section \ref{IFunc}). As $C$ is a fixed tail, this maps into the fixed component $F_{\star,\beta_0}$. Including constant map $f$ to $Y$ and taking the data over an arbitrary scheme, we have a morphism 
\[
F_{\beta}\rightarrow Y\times_{[W/S]} F_{\star, \beta_0}.
\]
Denote $F_{\star,\beta}:=Y\times_{[W/S]} F_{\star, \beta_0}$.

\begin{lemma}\label{lem32}
	The morphism $F_{\beta}\rightarrow F_{\star,\beta}$ is \'etale of degree $\mathfrak{r}_{\star}$.
\end{lemma}

\begin{proof}
	Lifting a morphism $S\rightarrow Y\times_{[W/S]} F_{\star, \beta_0}$ to $S\rightarrow F_{\beta}$ is the same as choosing an $\mathfrak{r}_{\star}$-th root of $v_{\infty}$. Note that the point in $Y$ completely determines the constant map $f$.
\end{proof}

\begin{lemma}\label{lem33}
The pullback of $[F_{\star,\beta}]^{\text{vir}}$ along the above morphism is equal to $[F_{\beta}]^{\text{vir}}$ and 
	\[
	\dfrac{1}{e_{\C^*}(N^{\text{vir}}_{F_{\beta}/MQ^{\epsilon_0}_{0,1}([E^{ss}/G],\beta)})}=(\mathfrak{r}_{\star}z)\cdot\mathbb{I}_{\beta_0}(\mathfrak{r}_{\star}z),
	\]
where 
	\[
	\mathbb{I}_{\tilde{\beta}_0}(z):= \dfrac{1}{e_{\C^*}(N^{\text{vir}}_{F_{\star,\beta_0}/QG_{0,1}(X^{ss},\beta_0)})}.
	\]
\end{lemma}

\begin{proof}
	Define $F_{\beta_0}$ to be the image of the forgetful morphism $F_{\beta}\rightarrow F_{\beta_0}$ forgetting the data of constant map $f:C\rightarrow Y$. Then we have 
	\[
	F_{\beta}=Y\times_{[W/S]} F_{\beta_0}.
	\]
	Note that $F_{\beta_0}$ is the fixed component of the master space $MQ^{\epsilon_0}_{0,1}(X^{ss},\beta_0)$ with target $X^{ss}$. Then the first result follows from \cite[Lemma 6.4.2]{YZ}. 
	
	The second expression follows by $MQ^{\epsilon_0}_{0,1}([E^{ss}/G],\beta)\cong Y\times_{[W/S]}MQ^{\epsilon_0}_{0,1}(X^{ss},\beta_0)$ and \cite[Lemma 6.4.2]{YZ}.
\end{proof}

\subsubsection{$v_1, v_2\neq 0$, Case II}
When $2g-2+n+\epsilon_0d>0$, the following lemma gives the $\C^*$-fixed components.

\begin{lemma}
Let 
	\[
	\xi=((\phi:C\rightarrow C_0,f,x),(C_0,x, [u], e,N,v_1,v_2))\in MQ^{\epsilon_0}_{g,n}([E^{ss}/G],\beta)(\C)
	\]
be an $\epsilon_0$-stable quasimap to $[E^{ss}/G]$ with calibrated tails such that $v_1\neq0$ and $v_2\neq0$. Then $\xi$ is $\C^*$-fixed if and only if 
\begin{enumerate}
	\item there is at least one degree-$(0,d_0)$ rational tail, and
	\item each entangled tail is a fixed tail.
\end{enumerate}	
\end{lemma}

\begin{proof}
As $v_1,v_2$ are nonzero, $v_1/v_2$ is a non-vanishing section of the calibration bundle. Let 
	\[
	\eta^{\sim}= ((\phi:C\rightarrow C_0,f,x),(C_0,x, [u], e))
	\]
be the underlying quasimap with entangled tails. By the construction of calibration bundle, $\xi$ is fixed if and only if the action of $\text{Aut}(\eta^{\sim})$ on the calibration bundle induces a surjection $\text{Aut}(\eta^{\sim})\twoheadrightarrow\C^*$. The result then follows by Lemma \ref{infaut}.
\end{proof}

For $k\geq1$, let $\overrightarrow{d}:=((d',{d_0''}),(0,{d}_1''),\cdots,(0,{d}_k''))$ be a $(k+1)$-tuple of pairs of non-negative integers such that $\text{deg}(\tilde{\beta})=(d',d_0'')+{(0,d_1'')}+\cdots +{(0,d''_k)}$ and ${d}_i''=d_0$ for $i=1,\cdots,k$.
Define $F_{\overrightarrow{d}}\subset MQ^{\sim}_{g,n}(X^{ss},\tilde{\beta})$ by 
\[
\begin{split}
	F_{\overrightarrow{d}}=\{\eta\vert\text{ }\eta \text{ has }& \text{exactly $k$ entangled tails} \\& \text{ which are fixed tails with degrees }(0,{d}''_1),\cdots,(0,{d}''_k)\}.
\end{split}
\]
From equation (\ref{eq10}), we have a projection map $MQ^{\sim}_{g,n}([E^{ss}/G],\beta)\xrightarrow{\text{pr}_2}MQ^{\sim}_{g,n}(X^{ss},(d',{\beta''}))$. Pulling back $F_{\overrightarrow{d}}$ along $\text{pr}_2$ and restricting it to $\epsilon_0$-stable master space, we have the fixed components 
\[
F_{\overrightarrow{\beta}}:=\text{pr}_2^*(F_{\overrightarrow{d}})\vert_{MQ^{\epsilon_0}_{g,n}([E^{ss}/G],\beta)},
\]
indexed by $(k+1)$-tuple of effective curves ${\overrightarrow{\beta}}:=(\overline{\beta},\beta_1,\cdots,\beta_k)$ such that $\beta=\overline{\beta}+\beta_1+\cdots+\beta_k$ with $\text{deg}(\beta_i)=(0,d_0)$ for each $i$ and $\text{deg}(\overline{\beta})=(d',d_0'')$.

\begin{lemma}
$F_{\overrightarrow{\beta}}\subset MQ^{\epsilon_0}_{g,n}([E^{ss}/G],\beta)^{\C^*}$ is an closed and open substack.
\end{lemma}

\begin{proof}
This follows by noting that $F_{\overrightarrow{d}}\subset MQ^{\epsilon_0}_{g,n}(X^{ss},(d',\tilde{\beta}))^{\C^*}$ is closed and open substack \cite[Lemma 6.5.2]{YZ}.
\end{proof}

Let $\mathfrak{C}^*_{k}\subset \tilde{\mathfrak{M}}_{g,n,d}$ be the locally closed smooth substack with exactly $k$ entangled tails. By \cite[Lemma 6.5.3]{YZ}, $F_{\overrightarrow{d}}\rightarrow\tilde{\mathfrak{M}}_{g,n,d}$ factors through $\mathfrak{C}^*_{k}$. Hence we have the following sequence of morphisms,
\[
F_{\overrightarrow{\beta}}\rightarrow F_{\overrightarrow{d}}\rightarrow\mathfrak{C}^*_{k}\rightarrow \mathfrak{Z}_{(k)}.
\]
Moreover, recall the gluing morphism (\ref{spliteq}) given by
\[
\tilde{\mathfrak{gl}_k}: \tilde{\mathfrak{M}}_{g,n+k,d-kd_0}\times'(\mathfrak{M}^{\text{wt,ss}}_{0,1,d_0})^k\rightarrow \mathfrak{Z}_{(k)}.
\]
Using the above two morphisms, we form a fibered diagram.
\begin{equation}
	\begin{tikzcd}
		{\tilde{\mathfrak{gl}_k}^*F_{\overrightarrow{\beta}}}\arrow[r]\arrow[d]&{\tilde{\mathfrak{gl}_k}^*F_{\overrightarrow{d}}}\arrow[r]\arrow[d]&{\tilde{\mathfrak{gl}_k}^*\mathfrak{C}^*_{k}}\arrow[r]\arrow[d]&{\tilde{\mathfrak{M}}_{g,n+k,d-kd_0}\times'(\mathfrak{M}^{\text{wt,ss}}_{0,1,d_0})^k}\arrow[d, "\tilde{\mathfrak{gl}_k}"]\\
		{F_{\overrightarrow{\beta}}}\arrow[r]&{F_{\overrightarrow{d}}}\arrow[r]&{\mathfrak{C}^*_{k}}\arrow[r]&{\mathfrak{Z}_{(k)}}.
	\end{tikzcd}
\end{equation}

Let $\mathcal{C}_{\tilde{\mathfrak{gl}_k}^*F_{\overrightarrow{\beta}}}$ be the universal curve over $\tilde{\mathfrak{gl}_k}^*F_{\overrightarrow{\beta}}$. Then $\mathcal{C}_{\tilde{\mathfrak{gl}_k}^*F_{\overrightarrow{\beta}}}$ can be obtained by gluing $\mathcal{E}_1,\cdots,\mathcal{E}_k$ to the last $k$-markings of $\mathcal{C}_{\overline{\beta}}$, where $$\mathcal{E}_1,\cdots,\mathcal{E}_k$$ is the pullback of universal curve of $(\mathfrak{M}^{\text{wt,ss}}_{0,1,d_0})^k$ and $\mathcal{C}_{\overline{\beta}}$ is the pullback of $\tilde{\mathfrak{M}}_{g,n+k,d-kd_0}$. Let $p_i\in \mathcal{C}_{\tilde{\mathfrak{gl}_k}^*F_{\overrightarrow{\beta}}}$  be the node on $\mathcal{E}_i$. Let $T_{p_i}\mathcal{E}_i$ be the line bundle formed by the relative tangent bundle along $p_i$ and denote $\Theta_i:=T_{p_i}\mathcal{E}_i\otimes T_{p_i}\mathcal{C}_{\overline{\beta}}$. Then by \cite[Lemma 2.5.5]{YZ} we have 
\[
\Theta_1\cong \cdots\cong \Theta_k=:\Theta.
\]

Let $ \mathbb{M}_{\overline{\beta}} $ be the calibration bundle on $\tilde{\mathfrak{M}}_{g,n+k,d-kd_0}$ and let $\mathbb{M}_{\overrightarrow{\beta}}$ be the calibration bundle on $F_{\overrightarrow{\beta}}$. Using a suitable pullback, we have a canonical isomorphism of line bundles on $\tilde{\mathfrak{gl}_k}^*F_{\overrightarrow{\beta}}$,
\[
\mathbb{M}_{\overline{\beta}}^{\vee} \otimes\Theta_1\otimes\cdots\otimes\Theta_k\cong\mathbb{M}_{\overrightarrow{\beta}}^{\vee}\overset{v_2/v_1}{\cong}\O_{\tilde{\mathfrak{gl}_k}^*F_{\overrightarrow{\beta}}}.
\]
Hence we have a canonical isomorphism on $\tilde{\mathfrak{gl}_k}^*F_{\overrightarrow{\beta}}$:
\begin{equation}\label{eq17}
	\Theta^{\otimes k}\cong \mathbb{M}^{\vee}_{\overline{\beta}}.
\end{equation}

Note that $\mathcal{C}_{\overline{\beta}}$ does not have any length-$d_0$ relevant base points. Therefore by restricting universal quasimaps to $\mathcal{C}_{\overline{\beta}}$, we have
\[
\tilde{\mathfrak{gl}_k}^*F_{\overrightarrow{\beta}}\rightarrow \tilde{Q}^{\epsilon_{+}}_{g,n+k}([E^{ss}/G],\overline{\beta}).
\]

Let 
\begin{equation}\label{eqn:Z_roots}
Z\rightarrow \tilde{Q}^{\epsilon_{+}}_{g,n+k}([E^{ss}/G],\overline{\beta}),
\end{equation}
be the stack of $k$-th roots of the pullback to $\tilde{Q}^{\epsilon_{+}}_{g,n+k}([E^{ss}/G],\overline{\beta})$ of the line bundle $\mathbb{M}^{\vee}_{\overline{\beta}}$ and let $L\to Z$ be the universal $k$-th root. Then (\ref{eq17}) gives rise to
\begin{equation}\label{eq18}
\tilde{\mathfrak{gl}_k}^*F_{\overrightarrow{\beta}}\rightarrow Z.
\end{equation}

Note that $\tilde{Q}^{\epsilon_{+}}_{g,n+k}([E^{ss}/G],\overline{\beta})\rightarrow \tilde{\mathfrak{M}}_{g,n+k,d-kd_0}$ factors through  $$\tilde{\mathcal{G}}^{\text{tw,}d'}_{g,n+k,\beta_0''}:={\mathcal{G}}^{\text{tw,}d'}_{g,n+k,\beta_0''}\times_{\mathfrak{M}^{\text{wt,ss}}_{g,n+k,\overline{d}}}\tilde{\mathfrak{M}}_{g,n+k,\overline{d}}.$$ 
 Define $Z_V\rightarrow \tilde{\mathcal{G}}^{\text{tw,}d'}_{g,n+k,\beta_0''}$ to be the stack of $k$-th roots of the pullback of  $\mathbb{M}^{\vee}_{\overline{\beta}}$. Then, there is a natural isomorphism 
\begin{equation}\label{eq21}
	Z\cong Z_V \times_{\mathfrak{Bun}^{\text{tw,wt}}_S}\mathfrak{M}^{\text{tw,pre}}_{g,n+k}(Y,\beta').
\end{equation}

Next we restrict our quasimaps to $\mathcal{E}_i$, namely the entangled tails. By definition, it consists of a constant map $f_i:\mathcal{E}_i\rightarrow Y$ and a quasimap with calibrated tails $[u_i]:\mathcal{E}_i\rightarrow X$ of degree $(0,d_0)$ and a base point of length $d_0$. This allows us to express $\tilde{\mathfrak{gl}_k}^*F_{\overrightarrow{\beta}}\vert_{\mathcal{E}_i} $ as a space over $Y$ with $\tilde{\mathfrak{gl}_k}^*F_{\overrightarrow{d}}\vert_{\mathcal{E}_i}$ as its fibers, 
\begin{align*}
\tilde{\mathfrak{gl}_k}^*F_{\overrightarrow{\beta}}\vert_{\mathcal{E}_i}  &\cong  Y\times_{[W/S]}\tilde{\mathfrak{gl}_k}^*F_{\overrightarrow{d}}\vert_{\mathcal{E}_i},\\
&\cong (Y\times_{[W/S]}I_{\mu}(X^{ss}))\times_{I_{\mu}(X^{ss})}\tilde{\mathfrak{gl}_k}^*F_{\overrightarrow{d}}\vert_{\mathcal{E}_i}\\
&\cong I_{\mu}([E^{ss}/G])\times_{I_{\mu}(X)}\tilde{\mathfrak{gl}_k}^*F_{\overrightarrow{d}}\vert_{\mathcal{E}_i}.
\end{align*}
As in \cite[Section 6.5]{YZ}, the goal is to relate the restrictions of quasimaps to $\mathcal{E}_i$ with fixed-domain quasimaps parametrized by the graph space. To do this, we need to construct a map from $\mathcal{E}_i$ to $\mathbb{P}^1$. Being a rational tail, $\mathcal{E}_i$ has two special points. This is not enough to canonically construct a map to $\mathbb{P}^1$. The solution to this, introduced in \cite[Section 6.5]{YZ}, is to include the choice of a tangent vector. This is done as follows.

Let $V^*_i\rightarrow Z$ be the total space of $L^{\otimes\mathfrak{r}_i}\otimes(T_{p_i}C_{\overline{\beta}}^{\otimes-\mathfrak{r}_i})$ minus its zero section. Define 
\[
Z':=V^*_1\times_Z \cdots \times_Z V^*_k,
\]
and twist $\mathcal{E}_i$ with it to get 
\begin{equation}\label{eqn:e-tail_i}
\mathcal{E}'_i:=\mathcal{E}_i\times_Z Z'=\mathcal{E}_{i}\times_{\tilde{\mathfrak{gl}_k}^*F_{\overrightarrow{\beta}}}(\tilde{\mathfrak{gl}_k}^*F_{\overrightarrow{\beta}}\times_Z Z').
\end{equation}
An $R$-point of $\tilde{\mathfrak{gl}_k}^*F_{\overrightarrow{\beta}}\times_Z Z'$ consists of a morphism $g:R\rightarrow \tilde{\mathfrak{gl}_k}^*F_{\overrightarrow{\beta}}$ along with non-vanishing sections $s_i\in H^0(R, g^*(T_{p_i}\mathcal{E}_i)^{\otimes\mathfrak{r}_i})$ for $i=1,\cdots, k$. And the family
\begin{equation*}
\mathcal{E}_i'|_R\to R    
\end{equation*}
obtained by pulling back $\mathcal{E}'_i\to \tilde{\mathfrak{gl}_k}^*F_{\overrightarrow{\beta}}$ via 
$g:R\rightarrow \tilde{\mathfrak{gl}_k}^*F_{\overrightarrow{\beta}}$ is the $i$-th entangled tail of the induced family of curves over $R$.

For each $i=1,\cdots, k$, 
$\mathcal{E}'_i$ carries the $\C^*\times (\C^*)^k$-action given by $\C^*$-action on $\mathcal{E}_i$ induced by its action on master space (and trivial on $ Z' $) and the $(\C^*)^k$-action given by scaling sections $(s_1,\cdots,s_k)$ (and acts trivially on $\mathcal{E}_i$). 

For each $i$, there is a unique morphism 
\begin{equation}\label{eqn:e-tail_i_P1}
g_i:\mathcal{E}'_i\rightarrow \P^1,
\end{equation}
which maps the marking $p_i$ to $\infty$, the unique base point to $0$, and the vector $s_i$ to $v_{\infty}$ (a fixed nonzero tangent vector to $\P^1$ at $\infty$). Moreover, by \cite[Lemma 6.5.4]{YZ}, the $\C^*\times(\C^*)^k$-action on $\mathcal{E}'_i$ is compatible with the standard $\C^*$-action on $\P^1$: Let $(\lambda,t)=(\lambda,t_1,\cdots, t_k)\in \C^*\times (\C^*)^k$, then the diagram
\begin{equation}\label{cd19}
	\begin{tikzcd}
		{\mathcal{E}'_i}\arrow[r, "g_i"]\arrow[d, "(\lambda^k\text{,}t)"']&{\P^1}\arrow[d,"\lambda^{\mathfrak{r}_i} t^{-1}_i"]\\
		{\mathcal{E}'_i}\arrow[r,"g_i"]&{\P^1}	
	\end{tikzcd}
\end{equation}
is commutative.

Recall the definition of quasimap graph space $QG_{0,1}(X^{ss},(0,d_i))$ and let $$QG^*_{0,1}(X^{ss},(0,d_i))\subset QG_{0,1}(X^{ss},(0,d_i))$$ be an open substack where the domain curve is irreducible. Define $$F_{*,d_i}\subset QG^*_{0,1}(X^{ss},(0,d_i))$$ to be the fixed (with respect to canonical $\C^*$-action) component where the marking
is at $\infty$ and a base point of length $d_0$ is at $0$.

Now consider the family of curves $\mathcal{E}'_i\rightarrow \tilde{\mathfrak{gl}_k}^*F_{\overrightarrow{\beta}}\times_Z Z'$ obtained in (\ref{eqn:e-tail_i}) together with the morphism $g_i:\mathcal{E}'_i\rightarrow \P^1$ obtained in (\ref{eqn:e-tail_i_P1}). This gives a morphism 
\[
\tilde{\mathfrak{gl}_k}^*F_{\overrightarrow{\beta}}\times_Z Z'\rightarrow QG_{0,1}(X^{ss},\tilde{\beta}_i),
\]
where deg$(\tilde{\beta}_i)=(0,d_i'')$. As $\mathcal{E}'_i$ is a fixed tail, the morphism maps into $\C^*$-fixed locus $F_{*,\tilde{\beta}_i}$. As $ F_{*,\tilde{\beta}_i} $ is $\C^*$-fixed, by diagram (\ref{cd19}) we get that the morphism is invariant under the action of $(\C^*)^k$. Hence the map descends to 
\[
\tilde{\mathfrak{gl}_k}^*F_{\overrightarrow{\beta}}\rightarrow  F_{*,\tilde{\beta}_i}.
\]
Combining the above map with the evaluation map at marked point $p_i$ corresponding to $\mathcal{E}_i$, we define a morphism
\begin{equation}\label{eq19}
	\tilde{\mathfrak{gl}_k}^*F_{\overrightarrow{\beta}}\rightarrow I_{\mu}([E^{ss}/G])\times_{I_{\mu}(X^{ss})}F_{*,\tilde{\beta}_i}=:F_{*,\beta_i},
\end{equation}
where the fiber product is formed by the evaluation map at special point $F_{*,\tilde{\beta}_i}\xrightarrow{ev} I_{\mu}(X^{ss})$ and the induced map $I_{\mu}([E^{ss}/G])\rightarrow I_{\mu}(X^{ss})$ on cyclotomic inertia stacks. 

For the stack of $k$-th roots $Z$ in (\ref{eqn:Z_roots}), composing with the evaluation map of $\tilde{Q}^{\epsilon_{+}}_{g,n+k}([E^{ss}/G],\overline{\beta})$ at the last $k$ markings gives
\[
ev_Z:Z\rightarrow (I_{\mu}[E^{ss}/G])^k
\]
Also consider the composition 
\[
ev_{*,\beta_i}:F_{*,\beta_i}\rightarrow I_{\mu}[E^{ss}/G]\rightarrow I_{\mu}[E^{ss}/G],
\]
where $F_{*,\beta_i}\rightarrow I_{\mu}[E^{ss}/G]$ is the projection map and $I_{\mu}[E^{ss}/G]\rightarrow I_{\mu}[E^{ss}/G]$ is the involution inverting the band. Using the above two maps, we can construct a fiber product
\[
Z\times_{(I_{\mu}[E^{ss}/G])^k}\prod_{i=1}^{k}F_{*,\beta_i}.
\]
By combining equations (\ref{eq18}) and (\ref{eq19}) we have a morphism
\begin{equation}\label{eq20}
\varphi: \tilde{\mathfrak{gl}_k}^*F_{\overrightarrow{\beta}}\rightarrow Z\times_{(I_{\mu}[E^{ss}/G])^k}\prod_{i=1}^{k}F_{*,\beta_i}.
\end{equation}

\begin{remark}\label{rem36}
	Note that, by definition of $F_{*,\beta_i}$, we can rewrite 
	\[
	Z\times_{(I_{\mu}[E^{ss}/G])^k}\prod_{i=1}^{k}F_{*,\beta_i}=Z\times_{(I_{\mu}(X^{ss}))^k}\prod_{i=1}^{k}F_{*,\tilde{\beta}_i}
	\]
	in terms of fixed locus data on graph quasimap to $X$.
\end{remark}

\begin{lemma}\label{lem34}
	The morphism $\varphi$ is representable, finite, \'etale, of degree $\prod_{i=1}^{k}\mathfrak{r}_i$.
\end{lemma}

\begin{proof}
Consider a faithfully flat 
cover $Z''\rightarrow Z$ over which $L$ and $T_{p_i}\mathcal{C}_{\overline{\beta}}$ are trivialized. Given an $R$-point $\xi$ of $Z''\times_{(I_{\mu}[E^{ss}/G])^k}\prod_{i=1}^{k}F_{*,\beta_i}$, we want to show that its lifting to $Z''\times_Z\tilde{\mathfrak{gl}_k}^*F_{\overrightarrow{\beta}}$ is equivalent to the choice of an $\mathfrak{r}_i$-th root of a nonzero section of certain line bundle.
	
Let $\xi\in Z''(R)\times_{(I_{\mu}[E^{ss}/G])^k}\prod_{i=1}^{k}F_{*,\beta_i}(R)$. Glue the special point of rational tail $\mathcal{E}''_i$ from $F_{*,\beta_i}(R)$ to the $(n+i)$-th marked point of the underlying curve $\mathcal{C}''_{\overline{\beta}}$ from $Z''(R)$. Similarly glue the quasimaps from both spaces. To get an $R$-point of $Z''\times_Z\tilde{\mathfrak{gl}_k}^*F_{\overrightarrow{\beta}}$, we need to choose the entanglement (i.e. a map $R\rightarrow\tilde{\mathfrak{M}}_{g,n,d}$) and define calibration bundle ($N,v_1,v_2$).
	
Note that the entanglement and calibration bundle data (by definition) depend completely on the underlying curve of the quasimap to $X^{ss}$. Hence the lemma follows from a similar result \cite[Lemma 6.5.5]{YZ} for quasimaps to $X^{ss}$.
\end{proof}

Next we extend the group actions in diagram (\ref{cd19}) over corresponding universal curves. Consider the  fibered diagram
\begin{equation}\label{cd22}
\begin{tikzcd}
	{Z'\times_Z\tilde{\mathfrak{gl}_k}^{*}F_{\overrightarrow{\beta}}}\arrow[r,"\varphi'"]\arrow[d,"p_1"]&{Z'\times_{(I_{\mu}(X^{ss}))^k}\prod_{i=1}^{k}F_{*,\tilde{\beta}_i}}\arrow[d,"p_2"]\\
	{\tilde{\mathfrak{gl}_k}^{*}F_{\overrightarrow{\beta}}}\arrow[r,"\varphi"]&{Z\times_{(I_{\mu}(X^{ss}))^k}\prod_{i=1}^{k}F_{*,\tilde{\beta}_i}}.
\end{tikzcd}
\end{equation}

Let $\mathcal{C}_1\to \tilde{\mathfrak{gl}_k}^{*}F_{\overrightarrow{\beta}}$ and $\mathcal{C}_2\to Z\times_{(I_{\mu}(X^{ss}))^k}\prod_{i=1}^{k}F_{*,\tilde{\beta}_i}$ be the universal curves. Pulling back, we get an isomorphism of universal curves 
\begin{equation}\label{eq23}
	\tilde{\varphi}:p_1^*\mathcal{C}_1\rightarrow p_2^*\mathcal{C}_2.
\end{equation}
We can glue $\C^*\times(\C^*)^k$-action on $\mathcal{E}_i$ for each $i=1,\cdots,k$ to get a $\C^*\times(\C^*)^k$-action on $p_1^*\mathcal{C}_1$. Moreover, analogous to diagram (\ref{cd19}), this action gives following commutative diagram: For any  $(\lambda,t)=(\lambda,t_1,\cdots,t_k)\in \C^*\times (\C^*)^k$, we have 
\begin{equation}\label{cd24}
\begin{tikzcd}
	{p_1^*C_1}\arrow[r,"\tilde{\varphi}"]\arrow[d,"(\lambda^k\text{,}t)"']&{p_2^*C_2}\arrow[d,"(\lambda^{\mathfrak{r}_1} t^{-1}_1\text{,}\cdots\text{,}\lambda^{\mathfrak{r}_k} t^{-1}_k\text{,}t)"]\\
	{p_1^*C_1}\arrow[r,"\tilde{\varphi}"]&{p_2^*C_2}.
\end{tikzcd}
\end{equation}

This defines a $\C^*\times(\C^*)^k$-action on $p_2^*\mathcal{C}_2$. To be precise, each  $\lambda^{\mathfrak{r}_i} t^{-1}_i$ acts by the standard $\C^*$-action on $\P^1$-component corresponding to $\mathcal{E}_i$ and $t\in (\C^*)^k$ acts by the standard scaling of sections of $Z'$. In conclusion, $\tilde{\varphi}:p_1^*\mathcal{C}_1\rightarrow p_2^*\mathcal{C}_2$ is a $\C^*\times(\C^*)^k$-equivariant map, where the $\C^*\times(\C^*)^k$-action on $p_2^*\mathcal{C}_2$ is given by vertical right arrow in the above diagram. 

Now we have all we need to prove an important lemma. Let $[Z]^{\text{vir}}\in A_*(Z)$ be the flat pullback of $[\tilde{Q}^{\epsilon_{+}}_{g,n+k}([E^{ss}/G],\overline{\beta})]^{\text{vir}}$ and let $[\tilde{\mathfrak{gl}_k}^*F_{\overrightarrow{\beta}}]^{\text{vir}}$ be the flat pullback of $[F_{\overrightarrow{\beta}}]^{\text{vir}}$. Let $\tilde{\psi}(\mathcal{E}_i)$ be the orbifold $\psi$-class of the rational tail $\mathcal{E}_i$ at the unique node and let $\tilde{\psi}_{n+i}$ be the orbifold $\psi$-class of $\tilde{Q}^{\epsilon_{+}}_{g,n+k}([E^{ss}/G],\overline{\beta})$ at the $(n+i)$-th marking. Let $\psi(\mathcal{E}_i)$ and $\psi_{n+i}$ be the corresponding coarse $\psi$-classes. Let ${\mathcal{D}_i}\subset\tilde{\mathfrak{M}}_{g,n,d}$ be the divisor defined by the closure of the locus where there are exactly ${i}$ entangled tails. Let $I_{\tilde{\beta}_i}(z)$ be as defined in Section \ref{IFunc} for the quasimap graph space $QG_{0,1}(X^{ss},\tilde{\beta}_i)$. Moreover, to simplify  notation, we define \[
\mathbb{I}_{\tilde{\beta}_i}(z):= \dfrac{1}{e_{\C^*}(N^{\text{vir}}_{F_{\star,\beta}/QG_{0,1}(X^{ss},\beta)})}.
\]
Note that $I_{\tilde{\beta}_i}(z)=\mathfrak{r}^2(\hat{\text{ev}})_*(\mathbb{I}_{\tilde{\beta}_i}(z)\cap[F_{\star,\beta}]^{\text{vir}})$.

\begin{lemma}\label{lem35}
Via the morphism $\varphi$ in (\ref{eq20}), we have
	\[
	[\tilde{\mathfrak{gl}_k}^*F_{\overrightarrow{\beta}}]^{\text{vir}}=\varphi^*( [Z]^{\text{vir}}\times_{(I_{\mu}[E^{ss}/G])^k}\prod_{i=1}^{k}[F_{*,\beta_i}]^{\text{vir}}),
	\]
and
	\[
	\begin{split}
		\dfrac{1}{e_{\C^*}(N^{\text{vir}}_{F_{\overrightarrow{\beta}}/MQ^{\epsilon_0}_{g,n}([E^{ss}/G],\beta)} \vert_{\tilde{\mathfrak{gl}_k}^*F_{\overrightarrow{\beta}}})} = \dfrac{\prod_{i=1}^{k}(\dfrac{\mathfrak{r}_i}{k}z+\psi(\mathcal{E}_i))}{-\dfrac{z}{k}-\tilde{\psi}(\mathcal{E}_1)-\tilde{\psi}_{n+1}-\sum_{i=k}^{\infty} |\mathcal{D}_i|}\cdot \boxtimes_{i=1}^{k}\mathbb{I}_{\tilde{\beta}_i}(\dfrac{\mathfrak{r}_i}{k}z+\psi(\mathcal{E}_i)).
	\end{split}\]
\end{lemma}

\begin{proof}
	Let $\mathbb{E}_{MQ}$ be the absolute perfect obstruction theory on $MQ^{\epsilon_0}_{g,n}([E^{ss}/G],\beta)$ (see Section \ref{obthrmaster}). We have a distinguished triangle induced by $MQ^{\epsilon_0}_{g,n}([E^{ss}/G],\beta)\rightarrow M\tilde{\mathfrak{M}}_{g,n,d}$:
	\[
	\mathbb{L}_{M\tilde{\mathfrak{M}}_{g,n,d}}\rightarrow\mathbb{E}_{MQ}\rightarrow\mathbb{E}_1\xrightarrow{+1}.
	\]
	Restricting to $\tilde{\mathfrak{gl}_k}^*F_{\overrightarrow{\beta}}$ and taking $\C^*$-fixed components yield a distinguished triangle on $\tilde{\mathfrak{gl}_k}^*F_{\overrightarrow{\beta}}$,
	\[
	(\mathbb{L}_{M\tilde{\mathfrak{M}}_{g,n,d}}|_{\tilde{\mathfrak{gl}_k}^*F_{\overrightarrow{\beta}}})^f\rightarrow(\mathbb{E}_{MQ}|_{\tilde{\mathfrak{gl}_k}^*F_{\overrightarrow{\beta}}})^f\rightarrow(\mathbb{E}_1)^f\xrightarrow{+1}.
	\]
	By an argument similar to \cite[Lemma 6.5.6]{YZ}, we have $$(\mathbb{L}_{M\tilde{\mathfrak{M}}_{g,n,d}}|_{\tilde{\mathfrak{gl}_k}^*F_{\overrightarrow{\beta}}})^f\cong\mathbb{L}_{\tilde{\mathfrak{M}}_{g,n+k,d-d_0k}}|_{\tilde{\mathfrak{gl}_k}^*F_{\overrightarrow{\beta}}}.$$
	Thus $[\tilde{\mathfrak{gl}_k}^*F_{\overrightarrow{\beta}}]^{\text{vir}}$ is defined by the relative perfect obstruction theory $(\mathbb{E}_1)^f$ over $\tilde{\mathfrak{M}}_{g,n+k,d-d_0k}$. Moreover, we have
	\begin{equation}\label{eq25}
	\dfrac{1}{e_{\C^*}((\mathbb{L}^{\vee}_{M\tilde{\mathfrak{M}}_{g,n,d}}|_{\tilde{\mathfrak{gl}_k}^*F_{\overrightarrow{\beta}}})^{\text{mv}})}=\dfrac{\prod_{i=1}^k(\dfrac{\mathfrak{r}_i}{k}z+\psi(\mathcal{E}_i))}{-\dfrac{z}{k}-\tilde{\psi}(\mathcal{E}_i)-\tilde{\psi}_{n+1}-\sum_{i=k}^{\infty}|\mathcal{D}_i|}.
	\end{equation}
    The formula (\ref{eq25}) for the equivariant Euler class is analogous to \cite[Equation (6.20)]{YZ}, whose proof can be adapted in a straightforward manner to prove (\ref{eq25}) using the above description of $\tilde{\mathfrak{gl}_k}^*F_{\overrightarrow{\beta}}$ and the moduli stack $M\tilde{\mathfrak{M}}_{g,n,d}$ in Definition \ref{sscabun} (which is the same as that in \cite{YZ}). In addition, weights of $\mathbb{C}^*$-actions involved in (\ref{eq25}) can be found in \cite[Proof of Lemma 6.5.6]{YZ}.
	
	Now we describe the obstruction theory for the right-hand side. By Remark \ref{rem36}, we have  $Z\times_{(I_{\mu}([E^{ss}/G]))^k}\prod_{i=1}^{k}F_{*,\beta_i}\cong Z\times_{(I_{\mu}(X^{ss}))^k}\prod_{i=1}^{k}F_{*,\tilde{\beta}_i}$.   Using equation (\ref{eq21}), a construction similar to Section \ref{obthrmaster} gives an obstruction theory for $Z$. By \cite[Section 6.3]{YZ} and a standard splitting-node argument, we have a perfect obstruction theory $\mathbb{E}_2$ for $Z\times_{(I_{\mu}([E^{ss}/G]))^k}\prod_{i=1}^{k}F_{*,\beta_i}$ relative to $\tilde{\mathfrak{M}}_{g,n+k,d-kd_0}$. Finally, the virtual cycle is given by the relative perfect obstruction theory $(\mathbb{E}_2)^f$.
	
	To relate $\mathbb{E}_1$ and $\mathbb{E}_2$, we pull back the morphism $\varphi$ to $\varphi'$ as in diagram (\ref{cd22}). By noting that the isomorphism $\tilde{\varphi}$ between universal curves (\ref{eq23}) commutes with maps to $Y$ and $X$, we have an isomorphism of obstruction theories
	\begin{equation}
		\alpha: p_1^*\mathbb{E}_1\cong p_1^*\varphi^*\mathbb{E}_2.
	\end{equation}
	As $\tilde{\varphi}$ is $\C^*\times (\C^*)^k$-equivariant, so is $\alpha$. As $\C^*$ acts trivially on $Z'\times_Z\tilde{\mathfrak{gl}_k}^{*}F_{\overrightarrow{\beta}}$, we obtain an isomorphism of $({\C^*})^k$-equivariant objects:
	\[
	\alpha^{\C^*}:(p_1^*\mathbb{E}_1)^{\C^*}\cong (p_1^*\varphi^*\mathbb{E}_2)^{\C^*}.
	\]
	As $(p_1^*\mathbb{E}_1)^{\C^*}=p_1^*\mathbb{E}_1^f$ and $(p_1^*\varphi^*\mathbb{E}_2)^{\C^*}=p_1^*\varphi^*\mathbb{E}_2^f$, we get an isomorphism between fixed parts
	\[
	\mathbb{E}_1^f\cong\varphi^*\mathbb{E}_2^f.
	\]
	This shows that they induce the same virtual fundamental class.
	
	Next we relate the moving parts of $\mathbb{E}_1$ and $\mathbb{E}_2$. Let 
	\[
	(z_1,\cdots,z_k,w_1,\cdots,w_k)
	\]
	be the $(\C^*)^k\times(\C^*)^k$-equivariant parameters. We obtain
	\[
	\dfrac{1}{e_{(\C^*)^k\times(\C^*)^k}(p_2^*(\mathbb{E}_2^{\vee\text{,mv}}))}=\mathbb{I}_{\tilde{\beta}_1}(z_1)\boxtimes\cdots\boxtimes \mathbb{I}_{\tilde{\beta}_k}(z_k).
	\]
	Relating the $(\C^*)^k$-action on $p_2^*\mathbb{E}_2$ to the $\C^*$-action on $p_1^*\mathbb{E}_1$ (see (\ref{cd24})), we have
	\[
	\dfrac{1}{e_{\C^*\times(\C^*)^k}(p_1^*(\mathbb{E}_1^{\vee\text{,mv}}))}=\mathbb{I}_{\tilde{\beta}_1}(\dfrac{\mathfrak{r}_1}{k}z-w_1)\boxtimes\cdots\boxtimes \mathbb{I}_{\tilde{\beta}_k}(\dfrac{\mathfrak{r}_k}{k}z-w_k).	
	\]
	By the canonical isomorphism between $\C^*$-equivariant intersection theory of $\tilde{\mathfrak{gl}_k}^*F_{\overrightarrow{\beta}}$ and the $\C^*\times(\C^*)^k$-equivariant  intersection theory of $Z'\times_Z \tilde{\mathfrak{gl}_k}^*F_{\overrightarrow{\beta}}$, we have 
	\[
     \dfrac{1}{e_{\C^*}(\mathbb{E}_1^{\vee\text{,mv}})}=\mathbb{I}_{\tilde{\beta}_1}(\dfrac{\mathfrak{r}_1}{k}z-\psi(\mathcal{E}_1))\boxtimes\cdots\boxtimes \mathbb{I}_{\tilde{\beta}_k}(\dfrac{\mathfrak{r}_k}{k}z-\psi(\mathcal{E}_k)).	
   	\]	
   	This along with equation (\ref{eq25}) gives us the result.
\end{proof}


\section{Wall-Crossing Formula}\label{sec:wallcrossing}
\subsection{Outline} 
Recall that $X=[\tilde{V}/(S\times G)]=[(W\times V)/(S\times G)]$. We assume the ring
\[
\bigoplus_{m=0}^{m=\infty}H^0({V},\O_{{V}}(m{\theta}))^{S\times G}
\]
of $(S\times G)$-invariants is generated by $H^0({V},\O_{{V}}({\theta}))^{S\times G}$ as an $H^0({V},\O_{{V}}))^{S\times G}$-algebra. This gives a map 
\begin{equation}\label{eq14}
	X\rightarrow [(W\times \C^{N+1})/(S\times \C^*)],
\end{equation}
where $N+1$ is the size of a basis for $H^0({V},\O_{{V}}({\theta}))^{S\times G}$. By composing (\ref{eq14}) with quasimaps to $X$, we have 
$$Q^{\text{pre}}_{g,n}([E^{ss}/G],\beta)\rightarrow  Q^{\text{pre}}_{g,n}(\tilde{\P},{(d',d'')}),$$ 
where $\tilde{\P}$ is a fiber bundle over $Y$ with $\P^N$ as its fiber (see Section \ref{sec2.9}). This along with diagram (\ref{cd2}) induce a map on $\epsilon_+$-stable quasimaps
\begin{equation}
	i:Q^{\epsilon_+}_{g,n}([E^{ss}/G],\beta)\rightarrow Q^{\epsilon_{+}}_{g,n}(\tilde{\P},(d',d'')).
\end{equation}
Since fibers of $\tilde{\P}$ are isomorphic to $\P^N$ and $f:C\to Y$ is constant on relevant rational tails, we can define a morphism (see Section \ref{sec2.9} for construction)
\[
c:Q^{\epsilon_{+}}_{g,n}(\tilde{\P},(d',d''))\rightarrow Q^{\epsilon_{-}}_{g,n}(\tilde{\P},(d',d''))
\]
by contracting all the degree-$(0,d_0)$ relevant rational tails to length-$d_0$ base points. Using these morphism, we have 
\begin{equation}\label{wcdiag}
\begin{tikzcd}
	{Q^{\epsilon_+}_{g,n}([E^{ss}/G],\beta)}\arrow[d,"i"]&{Q^{\epsilon_-}_{g,n}([E^{ss}/G],\beta)}\arrow[d,"i"]\\
	{Q^{\epsilon_{+}}_{g,n}(\tilde{\P},(d',d''))}\arrow[r, "c"]&{Q^{\epsilon_{-}}_{g,n}(\tilde{\P},(d',d''))}\arrow[r,"c_{\epsilon_{-}}"] &{Q^{0+}_{g,n}(\tilde{\P},(d',d''))},
\end{tikzcd}
\end{equation}
where $c_{\epsilon_{-}}$ is defined as in equation (\ref{eq11}).

The wall-crossing formula should relate two virtual fundamental classes $[Q^{\epsilon_+}_{g,n}([E^{ss}/G],\beta)]^{\text{vir}}$ and $[Q^{\epsilon_-}_{g,n}([E^{ss}/G],\beta)]^{\text{vir}}$. Since there is no natural morphism between these two moduli spaces, we use (\ref{wcdiag}) to push forward both classes to $Q^{0+}_{g,n}(\tilde{\P},(d',d''))$ and compare them there.

 By virtual localization formula \cite{Graber1997LocalizationOV}, \cite{CKL17}, we have
\begin{equation}\label{loceqn}
	[MQ^{\epsilon_0}_{g,n}([E^{ss}/G],\beta)]^{\text{vir}}=\sum_{*}(i_{F_{*}})_{*}\left( \dfrac{[F_*]^{\text{vir}}}{e_{\C^*}(N^{\text{vir}}_{F_*/MQ^{\epsilon_0}_{g,n}([E^{ss}/G],\beta)})}\right),
\end{equation}
where the sum is over all fixed components ($F_+,F_-$ and $F_{\overrightarrow{\beta}}$) and $i_{F_*}$ is the inclusion of corresponding  component into $MQ^{\epsilon_0}_{g,n}([E^{ss}/G],\beta)$.

Define a morphism 
\[
\tau: MQ^{\epsilon_0}_{g,n}([E^{ss}/G],\beta)\rightarrow Q^{0+}_{g,n}(\tilde{\P},(d',d''))
\]
by
\begin{itemize}
	\item composing the underlying quasimap $Q^{\sim}_{g,n}([E/G],\beta)\rightarrow Q^{\sim}_{g,n}(X,(d',{\beta''}))\rightarrow X$ with (\ref{eq14}),
	\item taking the coarse moduli of the domain curve, 
	\item taking the $0^+$-stablilization of the obtained quasimaps to $\tilde{\P}$. 
	
\end{itemize}

Consider the trivial $\C^*$-action on $ Q^{0+}_{g,n}(\tilde{\P},(d',d''))$. Then $\tau$ is $\C^*$-equivariant. Pushing forward (\ref{loceqn}), we have 
\begin{equation}\label{eq31}
\sum_{*}\tau_*(i_{F_{*}})_{*}\left( \dfrac{[F_*]^{\text{vir}}}{e_{\C^*}(N^{\text{vir}}_{F_*/MQ^{\epsilon_0}_{g,n}([E^{ss}/G],\beta)})}\right)=\tau_*[MQ^{\epsilon_0}_{g,n}([E^{ss}/G],\beta)]^{\text{vir}}.
\end{equation}
The right hand side lies in $A_*(Q^{0+}_{g,n}(\tilde{\P},(d',d'')))\otimes_{\Q}\Q[z]$, so the residue at $z=0$ of the left hand side is zero.

To simplify the notation, we suppress the obvious pushforward. In particular, for any moduli stack $M$ and morphism $\tau:M\rightarrow Q^{0+}_{g,n}(\tilde{\P},(d',d''))$, we write
\begin{equation}\label{notation:pushforward}
\int_{\beta}\alpha:=\tau_*(\alpha\cap\beta),\text{\hspace{1cm} for } \alpha,\beta\in A_*(M). 
\end{equation}

\subsection{Case \texorpdfstring{$g=0,n=1$}{g=0,n=1} { and }\texorpdfstring{$d=d_0$}{d=d0}} 
For the cyclotomic inertia stack 
\[
I_{\mu}([E^{ss}/G])=\coprod_{r}I_{\mu_r}([E^{ss}/G]),
\]
we have a locally constant function that takes value $r$ on $I_{\mu_r}([E^{ss}/G])$. Set $\mathfrak{r}_1:=\text{ev}_1^*(r)$.
\begin{lemma}
For $l=0,1,2,\cdots,$
$$
\int_{[Q^{\epsilon_{+}}_{0,1}([E^{ss}/G],\beta)]^{\text{vir}}}\mathfrak{r}_1^2\psi_1^l=\mathrm{Res}_{z=0}(z^{l+1}I_{\beta_0}(z)).
$$
\end{lemma}

\begin{proof}
	Let $\tilde{\psi}_1$ be the corresponding orbifold $ \psi $-class on the master space $MQ^{\epsilon_{+}}([E^{ss}/G],\beta)$. Applying localization formula to the master space, we have 
	\[
\begin{split}
		\int_{[MQ_{0,1}^{\epsilon_{0}}([E^{ss}/G],\beta)]^{\text{vir}}}\tilde{\psi}^l_1=\int_{[Q^{\epsilon_{+}}_{0,1}([E^{ss}/G],\beta)]^{\text{vir}}}\dfrac{\tilde{\psi}^l_1|_{Q^{\epsilon_{+}}_{0,1}([E^{ss}/G],\beta)}}{e_{\C^*}(N^{\text{vir}}_{Q^{\epsilon_{+}}_{0,1}([E^{ss}/G],\beta)/MQ^{\epsilon_{0}}_{0,1}([E^{ss}/G],\beta)})}\\
		+\int_{[F_{\beta}]^{\text{vir}}}\dfrac{\tilde{\psi}_1^l|_{F_{\beta}}}{e_{\C^*}(N^{\text{vir}}_{F_{\beta}/MQ^{\epsilon_{0}}_{0,1}([E^{ss}/G],\beta)})}.
\end{split}
	\]
	Note that for $g=0,n=1,d=d_0$, $Q^{\epsilon_{-}}_{0,1}([E^{ss}/G],\beta)$ is empty. By Section \ref{fixcomp1}, Lemmas  \ref{lem32}, and \ref{lem33}, we have 
	\[
	\begin{split}
		\int_{[MQ_{0,1}^{\epsilon_{0}}([E^{ss}/G],\beta)]^{\text{vir}}}\tilde{\psi}^l_1&=\int_{[Q^{\epsilon_{+}}_{0,1}([E^{ss}/G],\beta)]^{\text{vir}}}\dfrac{({\psi}^l_1/\mathfrak{r}_1)^l}{-z+\alpha}
		+\int_{[F_{\star,\beta}]^{\text{vir}}}\mathfrak{r}_1^2z^{l+1}\cdot(\mathbb{I}_{\beta_0}(\mathfrak{r}_1z))\\
		&=\int_{[Q^{\epsilon_{+}}_{0,1}([E^{ss}/G],\beta)]^{\text{vir}}}\dfrac{({\psi}^l_1/\mathfrak{r}_1)^l}{-z+\alpha}+z^{l+1}I_{\beta_0}(\mathfrak{r}_1z).
	\end{split}
	\]
Here $\alpha$ is the first Chern class of the calibration bundle on $Q^{\epsilon_{+}}_{0,1}([E^{ss}/G],\beta)$. Taking the residues on both sides gives 
	\[
	\int_{[Q^{\epsilon_{+}}_{0,1}([E^{ss}/G],\beta)]^{\text{vir}}}({\psi}^l_1/\mathfrak{r}_1)^l=\text{Res}_{z=0}(z^{l+1}I_{\beta_0}(\mathfrak{r}_1z)).
	\]
	Applying the change of variable $z\mapsto z/\mathfrak{r}_1$, we get the desired result.
\end{proof}

\subsection{Case \texorpdfstring{$2g-2+n+\epsilon_0d>0$}{2g-2+n+e0d>0}} 
By Lemma \ref{lem35}, contribution of $[F_{\overrightarrow{\beta}}]$ in the residue of the left hand side of equation (\ref{eq31}) is given by
\begin{equation}\label{eq32}
\begin{split}
\int_{[\tilde{\mathfrak{gl}}^*_kF_{\overrightarrow{\beta}}]^{\text{vir}}}\dfrac{\prod_{i=1}^k\mathfrak{r}_i}{k!}\text{Res}_{z=0}\left(\dfrac{\prod_{i=1}^k(\dfrac{\mathfrak{r}_i}{k}z+\psi(\mathcal{E}_i))}{-\dfrac{z}{k}-\tilde{\psi}(\mathcal{E}_1)-\tilde{\psi}_{n+1}-\sum_{i=k}^{\infty} [\mathcal{D}_i]}\cdot\right.  
\left. \boxtimes_{i=1}^{k}\mathbb{I}_{\tilde{\beta}_i}(\dfrac{\mathfrak{r}_i}{k}z+\psi(\mathcal{E}_i)) \vphantom{\dfrac{\dfrac{num}{den}}{{den}}}\right).
\end{split}
\end{equation}
Applying the change of variables 
\[
z\mapsto k(z-\tilde{\psi}(\mathcal{E}_1)-\tilde{\psi}_{n+1})=\cdots=k(z-\tilde{\psi}(\mathcal{E}_k)-\tilde{\psi}_{n+k})
\]
and using $\mathfrak{r}_i\tilde{\psi}_{n+i}=\psi_{n+i}$, (\ref{eq32}) becomes
\[
\begin{split}
	\int_{[\tilde{\mathfrak{gl}}^*_kF_{\overrightarrow{\beta}}]^{\text{vir}}}\dfrac{\prod_{i=1}^k\mathfrak{r}_i}{(k-1)!}\text{Res}_{z=0}\left(\dfrac{\prod_{i=1}^k(\mathfrak{r}_iz-\psi_{n+i})}{-z-\sum_{i=k}^{\infty} [\mathcal{D}_i]}\cdot\right.  
	\left. \boxtimes_{i=1}^{k}\mathbb
	{I}_{\tilde{\beta}_i}({\mathfrak{r}_i}z-\psi_{n+i}) \vphantom{\prod_1\dfrac{num}{\sum_i den}}\right).
\end{split}
\]
We push forward the expression along
\[
\tilde{\mathfrak{gl}}^*_kF_{\overrightarrow{\beta}}\xrightarrow{\varphi}Z\times_{(I_{\mu}(X^{ss}))^k}\prod_{i=1}^kF_{*,d_i}\xrightarrow{pr_Z}Z\rightarrow\tilde{Q}^{\epsilon_{+}}_{n+k}([E^{ss}/G],\overline{\beta}).
\]
Note that $\tilde{\mathfrak{gl}}^*_kF_{\overrightarrow{\beta}}\xrightarrow{\varphi}Z\times_{(I_{\mu}(X^{ss}))^k}\prod_{i=1}^kF_{*,d_i}$ has degree $\prod_{i=1}^k\mathfrak{r}_i$ by Lemma \ref{lem34} and $Z\rightarrow\tilde{Q}^{\epsilon_{+}}_{n+k}([E^{ss}/G],\overline{\beta})$ has degree $1/k$.
By \cite[Lemma 2.7.3]{YZ}, the pullback of $\mathcal{D}_i$ to $\tilde{\mathfrak{gl}}^*_kF_{\overrightarrow{\beta}}$ is equal to the pullback of boundary divisors $\mathcal{D}'_{i-k}$ of $\tilde{\mathfrak{M}}_{g,n+k,d-kd_0}$. By Lemma \ref{lem35} and the definition of $I$-coefficient, equation (\ref{eq32}) becomes
\begin{equation}
	\int_{[\tilde{Q}^{\epsilon_{+}}_{n+k}([E^{ss}/G],\overline{\beta})]^{\text{vir}}}\dfrac{1}{k!}\text{Res}_{z=0}\left(\dfrac{\prod_{i=1}^k\text{ev}^*_{n+1}((\mathfrak{r}_iz-\psi_{n+i}){I}_{\tilde{\beta}_i}({\mathfrak{r}_i}z-\psi_{n+i}))}{-z-\sum_{i=0}^{\infty} [\mathcal{D}'_i]}\right).
\end{equation}

\begin{remark}
	Recall that we define the integral by pushing it forward to $Q^{0+}_{g,n}(\tilde{\P},(d',d''))$. For $M=\tilde{Q}^{\epsilon_{+}}_{n+k}([E^{ss}/G],\overline{\beta})$, we push it forward by composing 
\[
\begin{split}
	\tilde{Q}^{\epsilon_{+}}_{n+k}([E^{ss}/G],\overline{\beta})\rightarrow {Q}^{\epsilon_{+}}_{n+k}([E^{ss}/G],\overline{\beta}) \xrightarrow{i}{Q}^{\epsilon_{+}}_{n+k}(\tilde{\P},(d',d''-kd_0))\xrightarrow{b_{\epsilon_{+},k}}\\ {Q}^{\epsilon_{+}}_{n}(\tilde{\P},(d',d''))\xrightarrow{c_{\epsilon_{+}}}{Q}^{0+}_{n}(\tilde{\P},(d',d'')).
\end{split}
\]
Here $b_{\epsilon_{+},k}$, $c_{\epsilon_{+}}$ are defined in Section \ref{sec2.9}.
\end{remark}

Next we  prove a lemma which will be used to integrate $\mathcal{D}'_i$.
\begin{lemma}\label{lem36}
	For $s\geq1,r=1,2,\cdots,$ we have 
\begin{equation}\label{eq34}
\begin{split}
&\int_{[\tilde{Q}^{\epsilon_{+}}_{n+k}([E^{ss}/G],\overline{\beta})]^{\text{vir}}}[\mathcal{D}'_{r}](\sum_{i=0}^{\infty}[\mathcal{D}'_i])^{s-1}\\
=&\sum_{\overline{\beta}'}\sum_{J}	\int_{[\tilde{Q}^{\epsilon_{+}}_{n+k+r}([E^{ss}/G],\overline{\beta}'')]^{\text{vir}}}\dfrac{(-1)^{s-r}}{r!}\cdot
\prod_{a=1}^r[\mathrm{ev}^*_{n+k+a}((\mathfrak{r}_{k+a}z-\psi_{n+k+a}){I}_{\overline{\beta}'_a}({\mathfrak{r}_{k+a}}z-\psi_{n+k+a}))]_{-j_a-1},
\end{split}
\end{equation}
where $\overline{\beta}'=(\overline{\beta}'',\overline{\beta}'_1,\cdots,\overline{\beta}'_r)$ runs through all $(r+1)$-tuples of effective curve classes such that $\overline{\beta}=\overline{\beta}''+\sum_{a=1}^{r}\overline{\beta}'_{a}$, deg($\overline{\beta}'_a)=(0,d_0)$ for $a=1,\cdots, r$; $J=(j_1,\cdots,j_r)$ runs through all $r$-tuples of non-negative integers such that $\sum_{a=1}^rj_a=s-r$.
\end{lemma}

\begin{proof}
	By Lemma \ref{lem22} we have
	\[
	\begin{split}
		\int_{[\tilde{Q}^{\epsilon_{+}}_{n+k}([E^{ss}/G],\overline{\beta})]^{\text{vir}}}[\mathcal{D}'_{r}](\sum_{i=0}^{\infty}[\mathcal{D}'_i])^{s-1}=
		\sum_{\overline{\beta}'}\dfrac{\prod_{a=1}^r\mathfrak{r}_{n+k+a}}{r!}\int_{Q_{\overline{\beta}'}}p_*((\sum_{i=0}^{\infty}[\mathcal{D}'_i])^{s-1}),
	\end{split}
	\]
	where $Q_{\overline{\beta}'}:=[\tilde{Q}^{\epsilon_{+}}_{g,n+k+r}([E^{ss}/G],\overline{\beta}'')]^{\text{vir}}\times_{(I_{\mu}[E^{ss}/G])^r}\prod_{a=1}^r[Q^{\epsilon_{+}}_{0,1}([E^{ss}/G],\overline{\beta}'_{a})]^{\text{vir}}$ and 
	\[
	p:\tilde{\mathfrak{gl}}_r^*\mathcal{D}'_r\times_{\tilde{\mathfrak{M}}_{g,n+k,d-kd_0}}\tilde{Q}^{\epsilon_{+}}_{g,n+k}([E^{ss}/G],\overline{\beta})\rightarrow \bigsqcup_{\overline{\beta}'}Q_{\overline{\beta}'}
	\]
	is the inflated projective bundle in the sense of Section \ref{InfProjBun} (by \cite[Lemma 2.7.2]{YZ} and Diagram (\ref{diag3.3})). Now by \cite[Lemma 2.7.4]{YZ}, \cite[Lemma A.0.1]{YZ}, and the projection formula, the left hand side of equation (\ref{eq34}) is equal to
	\[
	\sum_{\overline{\beta}'}\sum_{J}\dfrac{\prod_{a=1}^r\mathfrak{r}_{n+k+a}}{r!}\int_{Q_{\overline{\beta}'}}(-1)^{s-r}\prod_{a=1}^r(\tilde{\psi}_{n+k+a}+\tilde{\psi}_{a}')^{j_a},
	\]
	where $\tilde{\psi}_{n+k+a}$ is the orbifold $\psi$-class of $\tilde{Q}^{\epsilon_{+}}_{g,n+k+r}([E^{ss}/G],\overline{\beta}'')$ at the $(n+k+a)$-th marking, and $\tilde{\psi}_{a}'$ is the orbifold $\psi$-class of $Q^{\epsilon_{+}}_{0,1}([E^{ss}/G],\overline{\beta}'_{a})$.
	
	Integrating the powers of $\tilde{\psi}_{a}'$ against $[Q^{\epsilon_{+}}_{0,1}([E^{ss}/G],\overline{\beta}'_{a})]^{\text{vir}}$ using Lemma \ref{lem33}, the expression on the left hand side of (\ref{eq34}) becomes
	\[
	\sum_{\overline{\beta}'}\sum_{J}\int_{[\tilde{Q}^{\epsilon_{+}}_{g,n+k+r}([E^{ss}/G],\overline{\beta}'')]^{\text{vir}}}\dfrac{(-1)^{s-r}}{r!}\prod_{a=1}^r\sum_{b=0}^{j_a}\binom{j_a}{b} \tilde{\psi}_{n+k+a}^b\text{ev}^*_{n+k+a}[\mathfrak{r}zI_{\overline{\beta}'_{a}}(\mathfrak{r}z)]_{b-j_a-1}.
	\]
	Here we have suppressed the subscripts of $\mathfrak{r}$. Finally applying the change of variables $z\mapsto z-\tilde{\psi}$, we get the result.
\end{proof}

Expanding 
\[
\dfrac{1}{-z-\sum_{i=0}^{\infty} [\mathcal{D}'_i]}=\dfrac{-1}{z}+\sum_{s\geq 1}\sum_{r=1}^{\infty}(-z)^{-s-1}[\mathcal{D}'_{r}](\sum_{i=1}^{\infty}[\mathcal{D}'_i])^{s-1}
\]
and using Lemma \ref{lem36}, we get
\begin{Cor}\label{cor37}
	The contribution to the left hand side of (\ref{eq31}) from $F_{\overrightarrow{\beta}}$ is
\begin{equation}\label{eq35}
		\begin{split}
		-\dfrac{1}{k!}\sum_{r=0}^{\infty}\sum_{\overline{\beta}'}\sum_{\overline{b}}\dfrac{(-1)^r}{r!}\int_{[\tilde{Q}^{\epsilon_{+}}_{n+k+r}([E^{ss}/G],\overline{\beta}'')]^{\text{vir}}}[\prod_{i=1}^k\mathrm{ev}^*_{n+i}((\mathfrak{r}_iz-\psi_{n+i}){I}_{\tilde{\beta}_i}({\mathfrak{r}_i}z-\psi_{n+i}))]_{b_0}\cdot\\
		\prod_{a=1}^r\mathrm{ev}^*_{n+k+a}[(\mathfrak{r}_{k+a}z-\psi_{n+k+a}){I}_{\overline{\beta}'_a}({\mathfrak{r}_{k+a}}z-\psi_{n+k+a})]_{b_a},
	\end{split}
\end{equation}
where $\overline{\beta}'$ is similar to Lemma \ref{lem36} and $\overline{b}=(b_0,\cdots,b_r)$ runs through a $(r+1)$-tuple of integers such that $b_0+\cdots+b_r=0$ and $b_1,\cdots,b_r<0$.
\end{Cor}
 
Now we show the main theorem. Note that since we use the (abuse of) notation (\ref{notation:pushforward}), the following result is in fact an equality of {\em cycle classes}, not scalars.
\begin{thm}[Wall-crossing formula]\label{thm:wallcrossing}
	\[
\begin{split}
	\int_{[Q^{\epsilon_{-}}_{g,n}([E^{ss}/G],\beta)]^{\text{vir}}}&1-\int_{[Q^{\epsilon_{+}}_{g,n}([E^{ss}/G],\beta)]^{\text{vir}}}1\\
	&=\sum_{k\geq 1}\sum_{\overline{\beta}}\dfrac{1}{k!}\int_{[Q^{\epsilon_{+}}_{g,n+k}([E^{ss}/G],\overline{\beta})]^{\text{vir}}}\prod_{i=1}^{k}\mathrm{ev}^*_{n+i}[(z-\psi_{n+i})I_{\tilde{\beta}_i}(z-\psi_{n+i})]_0	
\end{split}
	\]
where $\overrightarrow{\beta}=(\overline{\beta},\beta_1,\cdots,\beta_k)$ runs through all the $(k+1)$-tuples of effective curve classes such that $\beta=\overline{\beta}+\beta_1+\cdots+\beta_k$ and $\text{deg}(\beta_i)=(0,d_0)$ for $i=1,\cdots,k$.
\end{thm}

\begin{proof}
	The sum of residues at $z=0$ of the left hand side of equation (\ref{eq31}) is zero. Contribution of each fixed term is given by Section \ref{fixcomp1} (and Lemma \ref{lem21}), Section \ref{fixcomp2}, and Corollary \ref{cor37}. Replace $k+r$ by $k$ in equation (\ref{eq35}) and rearrange the coefficient to get
	\[
	\begin{split}
		&\int_{[Q^{\epsilon_{-}}_{g,n}([E^{ss}/G],\beta)]^{\text{vir}}}1-\int_{[Q^{\epsilon_{+}}_{g,n}([E^{ss}/G],\beta)]^{\text{vir}}}1\\
		&-\sum_{k\geq 1}\sum_{\overline{\beta}}\sum_{r=0}^{k-1}\sum_{\overline{b}}\dfrac{(-1)^r}{r!(k-r)!}\int_{[Q^{\epsilon_{+}}_{g,n+k}([E^{ss}/G],\overline{\beta})]^{\text{vir}}}
		\prod_{i=1}^{k}\text{ev}^*_{n+i}[(\mathfrak{r}_iz-\psi_{n+i})I_{\tilde{\beta}_i}(\mathfrak{r}_iz-\psi_{n+i})]_{b_i}=0,
	\end{split}
	\]
	where $\overline{\beta}$ is as above and $\overline{b}=(b_1,\cdots,b_k)$ runs through all $k$-tuples such that $b_1+\cdots+b_k=0$ and $b_{k-r+1},\cdots, b_k<0$. Using the symmetry of the last $k$-markings, we can rewrite the expression as 
		\[
	\begin{split}
		&\int_{[Q^{\epsilon_{-}}_{g,n}([E^{ss}/G],\beta)]^{\text{vir}}}1-\int_{[Q^{\epsilon_{+}}_{g,n}([E^{ss}/G],\beta)]^{\text{vir}}}1=\\
		&\sum_{k\geq 1}\sum_{\overline{\beta}}\sum_{N\subsetneq\{1,\cdots,k\}}\sum_{\overline{b}}\dfrac{(-1)^{\#N}}{k!}\int_{[Q^{\epsilon_{+}}_{g,n+k}([E^{ss}/G],\overline{\beta})]^{\text{vir}}}
		\prod_{i=1}^{k}\text{ev}^*_{n+i}[(\mathfrak{r}_iz-\psi_{n+i})I_{\tilde{\beta}_i}(\mathfrak{r}_iz-\psi_{n+i})]_{b_i}=0.
	\end{split}
	\]
	Here $b_1+\cdots+b_k=0$  with $b_i<0$ for each  $i\in N$. For each fixed $k,\overline{b}$, note that 
	\[
	\sum_N(-1)^{\# N}=
	\begin{cases}
		1, & \text{if } b_i\geq 0 \text{ for all } i = 1,\cdots, k;\\
		0, & \text{otherwise},
	\end{cases}
	\]
	where the sum is over all $N\subsetneq\{1,\cdots,k\}$ such that $b_i<0$ for each $i\in N$. This along with $b_1+\cdots+b_k=0$ forces each $b_i$ that appear in the expression to be zero. Finally applying the change of variables $\mathfrak{r_i}z\mapsto z$ (which does not change the degree-$0$ term) we get the desired result.
\end{proof}

\begin{remark}
    Theorem \ref{thm:wallcrossing} may be extended to include insertions by working with moduli spaces with additional marked points, c.f. \cite[Remark 1.11.2]{YZ}. We leave the precise formulation of this extension to the interested readers.
\end{remark}

\bibliographystyle{alpha} 
\bibliography{myreference}
\end{document}